\documentclass[11pt]{amsart}

\usepackage[english]{babel}

\usepackage[a4paper,top=2cm,bottom=2cm,left=3cm,right=3cm,marginparwidth=1.75cm]{geometry}

\usepackage{amsmath}
\usepackage{graphicx}
\usepackage{comment}
\usepackage{mathtools}
\usepackage{fullpage}
\usepackage{tikz-cd}
\usepackage{nicefrac}
\usepackage{adjustbox}
\usepackage{leftindex}
\usepackage{color}
\usepackage[colorlinks=true, allcolors=blue]{hyperref}
\newtheorem{thm}{Theorem}[section]
\newtheorem{lemma}[thm]{Lemma}
\newtheorem{cor}[thm]{Corollary}
\newtheorem{prop}[thm]{Proposition}

\theoremstyle{definition}
\newtheorem{definition}[thm]{Definition}
\newtheorem{ex}[thm]{Example}
\newtheorem{idea}[thm]{Idea}

\theoremstyle{plain}
\newtheorem{rmk}[thm]{Remark}

\newtheorem{exercise}[thm]{Exercise}

\newtheorem{question}[thm]{Question}

\theoremstyle{definition}
\newtheorem{remark}[thm]{Remark}

\newcommand{\R}{\mathbb{R}}
\newcommand{\C}{\mathbb{C}}
\newcommand{\Z}{\mathbb{Z}}

\newcommand{\tb}{\widetilde{b}}

\newcommand{\Hom}{\operatorname{Hom}}

\newcommand{\Q}{\mathbb{Q}}

\newcommand{\mult}{\operatorname{mult}}
\newcommand{\glob}{\mathcal{O}}
\newcommand{\Sym}{\operatorname{Sym}}
\newcommand{\Gr}{\operatorname{Gr}}
\newcommand{\ind}{\operatorname{ind}}
\newcommand{\A}{\mathbb{A}}
\newcommand{\GW}{\operatorname{GW}}
\newcommand{\sheafHom}{\mathcal{H}\hspace{-0.1em}\mathit{om}}
\newcommand{\Puiseux}[2][t]{#2\{\!\{#1\}\!\}}
\newcommand{\trop}{\operatorname{trop}}
\newcommand{\Area}{\operatorname{Area}}
\newcommand{\Tr}{\operatorname{Tr}}
\newcommand{\rk}{\operatorname{rk}}
\newcommand{\sgn}{\operatorname{sgn}}
\renewcommand{\P}{\mathbb{P}}
\renewcommand{\phi}{\varphi}
\newcommand{\Bez}{\operatorname{Bez}}
\newcommand{\m}{\mathfrak{m}}
\newcommand{\Type}{\operatorname{Type}}
\newcommand{\val}{\operatorname{val}}
\newcommand{\tp}{\widetilde{p}}
\newcommand{\ta}{\widetilde{a}}
\newcommand{\tc}{\widetilde{c}}
\newcommand{\Wel}{\operatorname{Wel}}
\newcommand{\type}{\operatorname{type}}
\newcommand{\ev}{\operatorname{ev}}

\title{PCMI lecture notes: Motivic explorations in enumerative geometry}
\author{Sabrina Pauli}
\begin{document}
\maketitle

\begin{abstract}
    These are lecture notes for the PCMI 2024 Graduate Summer School for the mini-workshop on motivic explorations in enumerative geometry.
    Motivic homotopy theory allows to do enumerative geometry over an arbitrary field, which leads to additional arithmetic and geometric information. The goal of the mini-workshop is to explain why and how this works. We will also provide a toolbox for solving enumerative geometry problems in this setting, including the use of tropical geometry.

    We start with two classical examples in enumerative geometry, namely Bézout's theorem and the count of lines on a smooth cubic surface. We then explain how to solve these problems, first over the complex and real numbers, and then over an arbitrary field, using the $\A^1$-degree from motivic homotopy theory. 
    Then we introduce tropical geometry, more precisely we focus on tropical plane curves and show how they can be used to prove Bézout's theorem for curves over an arbitrary field. Finally, we discuss tropical correspondence theorems.
\end{abstract}

%Expected background and references:
%\begin{itemize}
%\item Some basic notions in intersection theory, characteristic classes and the implications to enumerative geometry \cite[Chapters 1-6]{3264}, \cite{MilnorStasheff}.
%\item The definition of the unstable motivic homotopy category. A well written survey article is \cite{WickelgrenWilliams}.
%\end{itemize}

\section*{Introduction}

These lecture notes are from the mini-workshop Motivic Explorations in Enumerative Geometry, 
held at the 2024 PCMI Graduate Summer School. The aim of the lectures is to show how problems in enumerative geometry can be extended to yield meaningful answers over an arbitrary base field, and how tropical geometry can be used to solve them in this setting.

Enumerative geometry is concerned with counting geometric objects, such as points, lines, or curves, 
that satisfy given algebraic conditions. Classical examples include B\'ezout's theorem, which counts 
the intersection points of hypersurfaces in projective space, and the fact that a smooth cubic surface 
always contains exactly $27$ lines \cite{cayley1849triple}. Over the complex numbers, these results are well understood and can 
be obtained by computing the degree of the top Chern class of a suitable vector bundle. 

Over the real numbers, one replaces the top Chern class with the Euler class, which also yields an enumerative 
invariant. The Poincar\'e--Hopf theorem then tells us that the degree of the Euler class equals the 
signed count of real objects. For instance, in the case of lines on a smooth cubic surface, the signed 
count of real lines is always $3$. Concretely, one can assign to each real line on a smooth real cubic surface a sign, 
either $+1$ or $-1$, in a consistent way, and the sum of these signs is always $3$, independent of the 
chosen surface \cite{Segre,FinashinKharlamovRealCount,OkonekTeleman}, while the actual number of real lines might be greater than $3$ \cite{Schlaefli}, depending on the chosen surface.

Motivic homotopy theory provides a unified framework that recovers both the classical complex counts 
and the real signed counts as special cases. In this setting, the weighted counts take values in the 
Grothendieck--Witt ring $\GW(k)$ of the base field $k$, producing richer invariants that additionally provide arithmetic information. In the two examples above, the top Chern class respectively Euler class 
is replaced by its motivic analogue. This motivic Euler class can be defined as a sum of local 
contributions, in direct analogy with the Poincar\'e--Hopf theorem \cite{KassWickelgrenCubicSurface,BachmannWickelgren}. These local contributions are 
elements of $\GW(k)$. 
There are also other, equivalent ways to define the Euler class in the motivic setting, see e.g. \cite{LevineAspects,BachmannWickelgren}, but we will focus on the approach via the Poincar\'e--Hopf theorem in these lecture notes.

We next introduce a new approach to solving such problems, namely through tropical geometry. 
In tropical geometry one studies piecewise-linear objects, and in tropical enumerative geometry 
one counts these objects. This is often much simpler than in the algebraic setting, since tropical 
objects are of combinatorial nature. 
\emph{Tropicalization} transforms algebraic objects into tropical ones. One can think of this as a degeneration: 
from an algebraic object we extract the combinatorial data of its limit, which then defines the 
corresponding tropical object. Remarkably, tropicalization retains enough information that many 
theorems can be solved on the tropical side, and then transferred back to 
algebraic geometry. 

In these lecture notes we will apply this idea to B\'ezout's theorem for curves, first in the 
classical setting and then over an arbitrary field with values in $\GW(k)$ \cite{JaramilloPuentesPauliBezout}. %For this particular 
%problem tropical geometry is not strictly necessary, unlike in the next example.

In the final part of these lectures we turn to a breakthrough in the use of tropical methods in 
enumerative geometry, namely Mikhalkin's tropical correspondence theorem \cite{Mikhalkin}. 
This result states that 
the number of plane genus $g$ curves of degree $d$ passing through $3d+g-1$ points in general position is equal to the 
corresponding weighted count of tropical curves, where each tropical curve is counted with its appropriate 
multiplicity. 
For genus $0$ there is a well defined weighted count of plane curves with point conditions with weights in $\GW(k)$ \cite{LevineWelschinger,KassLevineSolomonWickelgren}. 
We review very recent work showing that a tropical correspondence holds over a perfect field of characteristic $0$ or with large characteristic, with tropical multiplicities taking values in $\GW(k)$ \cite{JaramilloPuentesPauliCorrespondence,JPMPRnew}.

\subsection*{Outline}
We begin with classical and real enumerative geometry and then explain how motivic methods extend 
these ideas to arbitrary fields, illustrating the techniques through B\'ezout's theorem and the problem 
of counting lines on a smooth cubic surface (lecture 1).  Next, we introduce the Grothendieck--Witt ring 
$\GW(k)$ of a field $k$, provide concrete examples, and review some of its key properties, followed by 
formulas for the local $\A^1$-degree that are essential for our computations (lecture 2). We then turn to tropical 
geometry, focusing on tropical plane curves, and show how tropical methods can be used to prove a 
version of B\'ezout's theorem for curves (lecture 3). Finally, we complete the discussion by proving B\'ezout's 
theorem over an arbitrary field on the tropical side and conclude with tropical correspondence theorems (lecture 4).

\subsection*{Ackknowledgements}
I would like to thank Nathan Tiggemann for catching many typos and suggesting improvements and Marc Levine for a correction.

I would also like to thank the anonymous referee for reading this document thoroughly and providing such detailed corrections.

I acknowledge support by Deutsche Forschungsgemeinschaft (DFG, German Research Foundation) through the Collaborative
Research Centre TRR 326 Geometry and Arithmetic of Uniformized Structures, project number 444845124.

{\hypersetup{linkcolor=black}\tableofcontents}
\setcounter{tocdepth}{1}

\section{Enumerative geometry over different fields}
%Content of the Lecture:
%\begin{itemize}
%    \item First examples in enumerative geometry: Bézout's theorem and the count of lines on a smooth cubic surface.
    
%    Literature: \cite{3264}
%    \item The real story: The Poincaré-Hopf theorem and its applications to Bézout's theorem and the real signed count of lines on a smooth cubic surface.
    
%    Literature: \cite{MilnorStasheff,MilnorDifferentialViewpoint,OkonekTeleman,FinashinKharlamovRealCount}
%    \item The Poincaré-Hopf theorem in motivic homotopy theory.
    
%    Literature: %\cite{KassWickelgrenCubicSurface,BachmannWickelgren,McKeanBezout}
%\end{itemize}

In the first lecture, we start with reviewing two examples from classical enumerative geometry: 
B\'ezout's theorem, which counts the intersection points of $n$ hypersurfaces in projective $n$-space, 
and the count of lines on a smooth cubic surface. 
Both of these results can be obtained by computing the degree of the top Chern class of a suitable vector bundle. 
For the necessary background, see for example \cite{3264}.

Next, we move to the real versions of these problems. 
Here, the top Chern class is replaced by the Euler class, 
and the Poincar\'e--Hopf theorem allows us to interpret this as a signed count of real objects. 
Relevant references for this part are \cite{MilnorStasheff,MilnorDifferentialViewpoint,OkonekTeleman,FinashinKharlamovRealCount}.

Finally, we generalize these ideas to work over an arbitrary base field $k$, 
where the classical count ($k = \mathbb{C}$) and the real signed count ($k = \mathbb{R}$) appear as special cases. 
The tools for this come from motivic homotopy theory. 
Relevant literature can be found in \cite{KassWickelgrenCubicSurface,BachmannWickelgren,McKeanBezout}.

\subsection{Classical Enumerative Geometry}
In enumerative geometry, we want to count geometric objects, such as points, lines, curves,... that satisfy certain algebraic conditions.
An example of a result in enumerative geometry is one that everyone learns in the first course in algebraic geometry:
\begin{ex}[Bézout's theorem]
    For $i=1,\ldots,n$ let $H_i=V(F_i)\subset \mathbb{P}^n_\C$ be hypersurfaces defined by a homogeneous polynomial $F_i$ of degree $d_i$, such that $H_1\cap\ldots\cap H_n$ is zero-dimensional.
    Then
    \[\sum_{x\in H_1\cap\ldots\cap H_n}\mult_x(H_1,\ldots,H_n)=d_1\cdot \ldots \cdot d_n. \]
    In words, counted with intersection multiplicity, there are always $d_1\cdot \ldots \cdot d_n$ intersection points.
\end{ex}
Here is another famous example.
\begin{ex}[Lines on a smooth cubic surface]
    Let $V(f)\subset \mathbb{P}^3_\C$ be a smooth cubic surface (so $f$ is a homogeneous polynomial of degree $3$). Then 
    \[\#\{\ell\subset V(f): \; \text{$\ell$ a line}\}=27\]
    regardless of the choice of smooth cubic surface.
\end{ex}

Often one can solve problems in enumerative geometry by ``linearization", i.e., by computing degrees of chern classes. We illustrate this for the two examples above.

\begin{ex}[Bézout's theorem continued]
\label{ex:bezoutvb}
    Let $p\colon V\coloneqq \glob(d_1)\oplus\ldots \oplus \glob(d_n)\rightarrow \mathbb{P}^n_\C\eqqcolon X$. Then $\operatorname{rank}V=\operatorname{dim}X$ and thus a general section of this vector bundle has finitely many zeros. Now notice that an $n$-tuple $(F_1,\ldots,F_n)$ of homogeneous polynomials of degree $d_1,\ldots,d_n$ defines a section $\sigma_{F_1,\ldots,F_n}$ of $p\colon V\rightarrow X$ and the zeros of the section are exactly the intersections of the hypersurfaces $H_1,\ldots,H_n$. Thus
    \[\sum_{x\in H_1\cap\ldots\cap H_n}\mult_x(H_1,\ldots,H_n)=\sum_{\text{zeros $z$ of $\sigma_{F_1,\ldots,F_n}$}}\mult(z)\]
    where $\mult(z)$ is the multiplicity of the zero $z$ of $\sigma_{F_1,\ldots,F_n}$.
    For a general section of a vector bundle of rank $n$ over an $n$-dimensional scheme, the top chern class is represented by the $0$-cycle of the zero locus of a general section and thus $\sum_{\text{zeros $z$ of $\sigma_{F_1,\ldots,F_n}$}}\mult(z)=\deg c_n(V)$ where $c_n(V)$ is the top chern class.
    By the Whitney sum formula
    \[\deg c_n(V)=\deg c_1(\glob(d_1))\cdot\ldots\cdot\deg c_1(\glob(d_n))=d_1\cdot \ldots\cdot d_n.\]
\end{ex}

\begin{ex}[Lines on a smooth cubic surface continued]
\label{ex:linescubicsurface}
Let $\Gr(2,4)=\mathbb{G}(1,3)$ be the Grassmannian of lines in $\mathbb{P}^3_\C$,
that is $[\ell]\in \Gr(2,4)$ corresponds to a line $\ell\subset \mathbb{P}_\C^3$.
Further let $\mathcal{S}\rightarrow \Gr(2,4)$ be the tautological bundle.
    Let $p\colon V=\Sym^3\mathcal{S}^*\rightarrow \Gr(2,4)\eqqcolon X$ be the vector bundle given by the third symmetric power of the dual tautological bundle, that is the vector bundle with fibers 
    \[(\Sym^3\mathcal{S}^*)_{[\ell]}=\text{homogeneous degree $3$ polynomials on $\ell$.}\]
    Then a smooth cubic surface $V(f)\subset \mathbb{P}^3_\C$ defines a section $\sigma_f$ of $p\colon V\rightarrow X$, namely the section defined by restricting $f$ to the line, i.e. $\sigma_f([\ell])=f\vert_\ell$. Then
    \[\ell\subset V(f)\Leftrightarrow f\vert_l=0\Leftrightarrow \sigma_f([\ell])=0\]
    and so
    \[\#\{\ell\subset V(f):\; \text{$\ell$ a line}\}=\#\{\text{zeros of $\sigma_f$}\}\]
    One can show that for a smooth cubic surface defined by $f$ all zeros of $\sigma_f$ are simple, i.e. they are zeros of multiplicity $1$. Thus the actual number of zeros of a general section equals the degree of the top chern class
    $\deg c_4(V)$.
    One can use Schubert calculus to compute $\deg c_4(V)=27$.
\end{ex}

\subsection{Real Enumerative Geometry}
In real enumerative geometry one is interested in the geometric objects which are defined over the real numbers.

\begin{ex}[Real lines on a smooth cubic surface]
    On a smooth cubic surface defined by a polynomial $f$ with real coefficients, there can be $3$, $7$, $15$ or $27$ real lines depending on the choice of cubic surface \cite{Schlaefli}.
\end{ex}

This example shows that we lose \emph{invariance} when we count over non-algebraically closed fields. We already observed this in high school: A polynomial $p\in \R[x]$ (in one variable) of degree $d$ has $d$ zeros (if you count with multiplicity) defined over $\C$, but over $\R$ it can happen that there are none if $d$ is even, or only $1$ if $d$ is odd, but it can also happen that all zeros are real.
However, we still want to have an invariant count over non-algebraically closed fields. %Over $\R$ this can be done following the following idea.
\begin{idea}
    Replace the top chern class in the examples with the Euler class $e(V)$.
\end{idea}

\begin{ex}[Real lines on a smooth cubic surface]
There is a ``real version" of the vector bundle from Example \ref{ex:linescubicsurface} $p\colon \Sym^3\mathcal{S}^*\rightarrow \Gr_\R(2,4)$ where $\Gr_\R(2,4)$ is the real Grassmann manifold of $2$-planes in $\R^4$. 
The degree of the Euler class of this real vector bundle equals
    \[\deg e(\Sym^3\mathcal{S}^*\rightarrow \Gr(2,4))=3.\]
\end{ex}
But what does the $3$ tell us? Definition \ref{def:relativeorientationandlocalindex} and Theorem \ref{thm:PHtheorem} (the Poincaré-Hopf theorem) will tell us.
For this recall the degree of a map from the $n$-sphere to the $n$-sphere from algebraic topology.
Let $f\colon S^n\rightarrow S^n$ be continuous. Then $f$ induces $f_*\colon \widetilde{H}_n(S^n)\cong \Z\rightarrow \widetilde{H}_n(S^n)\cong \Z$ a map in the $n$th reduced singular homology. The degree of $f$ is the image of $1$ under this map $\deg f=f_*(1)$. When $f\simeq g\colon S^n\rightarrow S^n$ are homotopic, they induce the same map in homology and thus have the same degree. So we get
\begin{equation}
\label{eq:degree}
\deg \colon [S^n,S^n]\rightarrow \Z\end{equation}
where $[S^n,S^n]$ denotes the homotopy classes of maps from $S^n$ to $S^n$. The degree map $\deg$ is in fact an isomorphism.

\begin{definition}
\label{def:relativeorientationandlocalindex}
Let $p\colon V\rightarrow X$ be a real vector bundle of rank $n$ over a real smooth $n$-manifold.
\begin{itemize}
     \item 
     Let $T_X$ be the tangent bundle of $X$.
        A \emph{relative orientation} of $p\colon V\rightarrow X$ is an isomorphism of line bundles $\rho\colon \sheafHom(\det T_X,\det V)\cong \underline{\R}$, where $\underline{\R}$ denotes the trivial line bundle. We say that a vector bundle $p\colon V\rightarrow X$ is \emph{relatively orientable} if a relative orientation exists, and that the vector bundle is \emph{relatively oriented} if we have chosen a relative orientation.
        \item Assume $p\colon V\rightarrow X$ is relatively oriented with relative orientation $\rho$. Let $x\in X$ and choose a small neighborhood $U$ of $x$ such that there is a local trivialization of $V\vert_U$. We say that local coordinates around $x$ and a trivialization $V\vert_U\cong U\times \R^n$ are \emph{compatible} with $\rho$ if the distinguished section of $\sheafHom(\det T_X\vert_U,\det V\vert_U)$ taking the distinguished basis of $\det T_X\vert_U$ to $\det V\vert_U$ is sent to $1$ by $\rho$. 
          \item Assume $p\colon V\rightarrow X$ is relatively oriented by $\rho$ and $\sigma\colon X\rightarrow V$ is a section. Let $x$ be an isolated zero of $\sigma$. Locally around $x$ we can choose coordinates and a trivialization of $V$ compatible with $\rho$. Then the section $\sigma$ is given locally by $\sigma\colon \R^n\rightarrow \R^n$ ($x$ corresponds to $0$ in the source) in these coordinates and with this trivialization. The \emph{local index} $\ind_x\sigma$ of $\sigma$ at $x$ is the local degree of $\sigma$ at $0$, i.e. it is determined by
          \[\sigma_*\colon H_n(\R^n,\R^n\setminus \{0\})\cong \Z\xrightarrow{\cdot \ind_x\sigma} H_n(\R^n,\R^n\setminus \{0\})\cong \Z\]
        or equivalently it is the degree of
     \[\overline{\sigma}\colon \nicefrac{B_\epsilon(x)}{\partial B_\epsilon(x)} \simeq S^n\rightarrow \nicefrac{B_{\epsilon'}(0)}{\partial B_{\epsilon'}(0)}\simeq S^n \]
     where $B_\epsilon(x)\subset X$ is a small ball around $x$ and $B_{\epsilon'}(0)\subset \R^n$ is a small ball around $0$.
\end{itemize}  
     
\end{definition}

\begin{rmk}
\label{rmk:simplezeroofasection}
    In case $x$ is a simple zero of $\sigma$, that is a zero with multiplicity $1$, then locally $\sigma$ is a homeomorphism and thus $\ind_x\sigma\in \{\pm1\}$. In this case $\ind_x\sigma=\det \operatorname{Jac}\sigma(0)$ where we take the Jacobian with respect to the local coordinates and trivialization compatible with the relative orientation as in Definition \ref{def:relativeorientationandlocalindex}.
\end{rmk}
\begin{thm}[Poincaré-Hopf theorem]
\label{thm:PHtheorem}
    Let $X$ be a real smooth closed $n$-manifold and $p\colon V\rightarrow X$ be a real relatively oriented vector bundle of rank $n$. Let $\sigma$ be a section of $p\colon V\rightarrow X$ with only isolated zeros. Then
    \[\deg e(V)=\sum_{\text{zeros $x$ of $\sigma$}}\ind_x\sigma.\]
\end{thm}

\begin{ex}[Lines on a smooth cubic surface continued]
\label{ex:reallines}
If the cubic surface $V(f)$ is smooth, then all zeros of the corresponding section $\sigma_f$ are simple. So the local indices are either $+1$ or $-1$. 
That is, we get 
\[3=\deg e(V)=\sum_{\ell\subset V(f)}\ind_{[\ell]}\sigma_f\]
an invariant signed count of real lines on a smooth real cubic surface.
In fact, one can describe the sign of a real line intrinsically, i.e. without choosing a local trivialization and local coordinates, and this has already been done by Segre \cite{Segre} in the first half of the 20th century: Take a hyperplane in $\mathbb{RP}^3$ containing the line and intersect it with the cubic surface. The intersection is a degree-$3$ curve, hence the union of the line and a conic. This conic meets the line in two points, which defines an involution on the line (by exchanging the intersection points as the hyperplane varies). Segre classified the line as \emph{hyperbolic} or \emph{elliptic} depending on whether the involution has real fixed points, i.e. whether the intersection points are defined over $k$. One always has
\[
\#\text{hyperbolic lines} - \#\text{elliptic lines} = 3,
\]
independent of the chosen cubic surface as observed in \cite{OkonekTeleman}.

Geometrically one can thing of this like this: 
Note that the real projective line $\mathbb{RP}^1 \cong S^1$ is topologically a circle. Orient this circle.
As you move along this circle inside the cubic surface and choose a continuous frame of the tangent bundle so that one vector always points along the circle, 
the frame will either make a full turn or come back to where it started without turning, just moving back and forth \cite{BenedettiSilhol}.
In the first case the line has local index $-1$ and in the second case it has local index $+1$. So the local index depends on how the line is embedded in the cubic surface.
\end{ex}

\begin{ex}[Bézout continued]
\label{ex:bezoutreal}
    The real analog $p\colon V=\glob(d_1)\oplus\ldots\oplus \glob(d_n)\rightarrow \R\P^n $ of the vector bundle from the Bézout example \ref{ex:bezoutvb} is relatively orientable if and only if $\sum d_i\equiv n+1\mod 2$. 
    
    For example, if $n=1$, this is the case if $d_1$ is even. 
    In this case, Bézout's theorem counts the zeros of a polynomial in one variable. The Euler class of an odd rank bundle is zero in the real setting. If the derivatives at the zeros do not vanish, then the local index is given by the sign of the derivative by Remark \ref{rmk:simplezeroofasection} and thus the sum of these signs is always zero (see the left picture in Figure \ref{fig:realbezout}), while the number of real zeros depends on the polynomial.

   For $n=2$ the vector bundle is relatively orientable if $d_1+d_2$ is odd. In this case, the local index at a real intersection records the order in which the two curves intersect after orienting them (see the right image in Figure \ref{fig:realbezout}). It can be computed as the sign of the determinant of the Jacobian at the intersection. Since $V$ has an odd rank summand, applying the Whitney sum formula, we get that the degree of the Euler class of $V$ is zero, and thus the signed count of intersection points is zero. 

   In fact, for any $n$ and $d_1,\ldots,d_n$ such that $\sum d_i\equiv n+1\mod 2$, the signed count of intersection points will be zero since the corresponding Euler class is zero by the same argument: One computes the degree of the Euler class of a direct sum of line bundles. By the Whitney sum formula, this Euler class is the product of the Euler classes of the individual line bundles. However, each of these factors vanishes, since a line bundle has odd rank, and hence its Euler class is zero.
   If $\sum d_i \not\equiv n+1 \pmod{2}$, then no relative orientation exists. Consequently, one cannot define the degree of the Euler class, nor can one obtain well-defined local indices. This means we do not get a meaningful answer to the real signed counting problem.

\begin{figure}
\begin{tabular}{cc}
\includegraphics[scale=0.13]{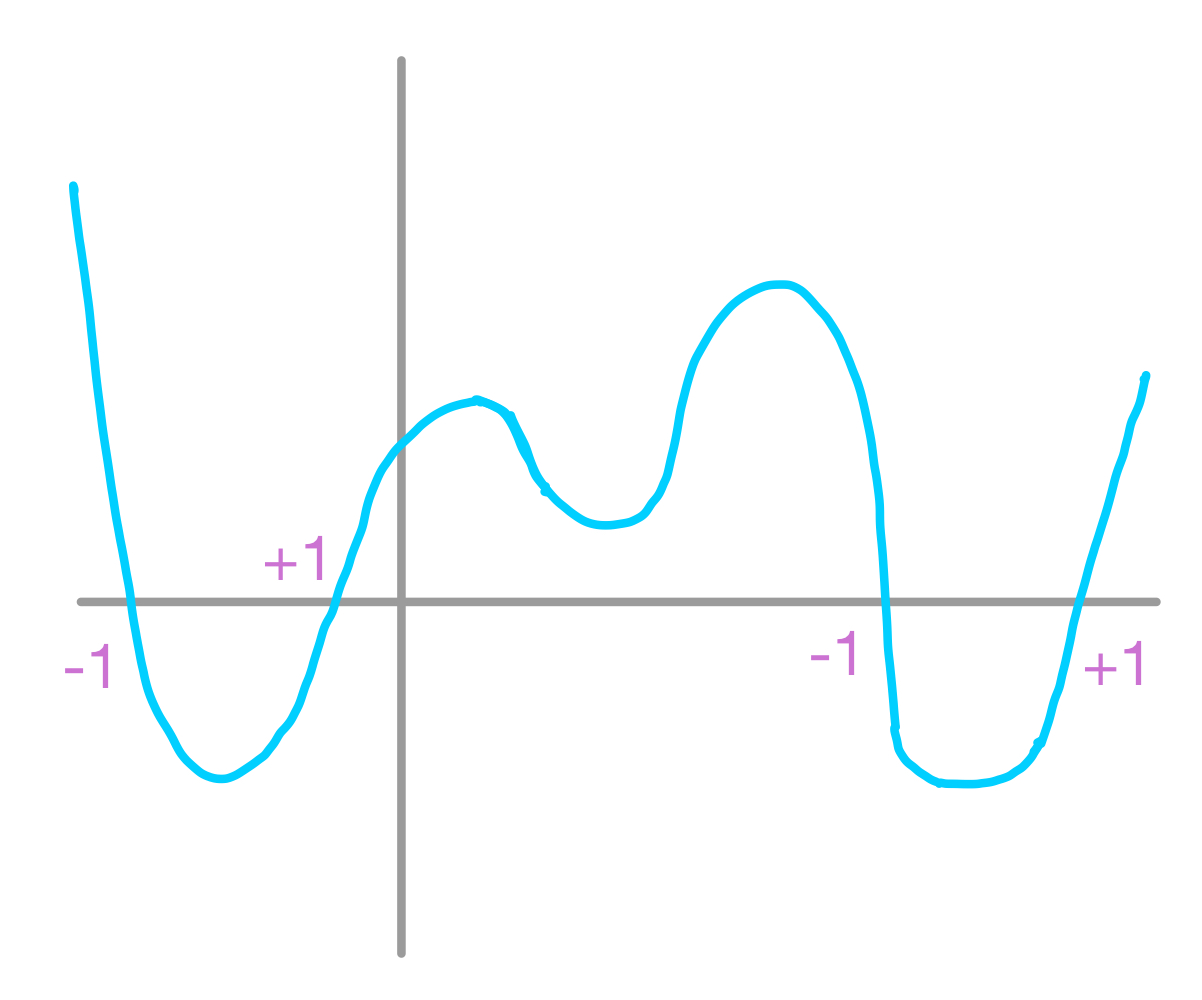}
&\includegraphics[scale=0.1]{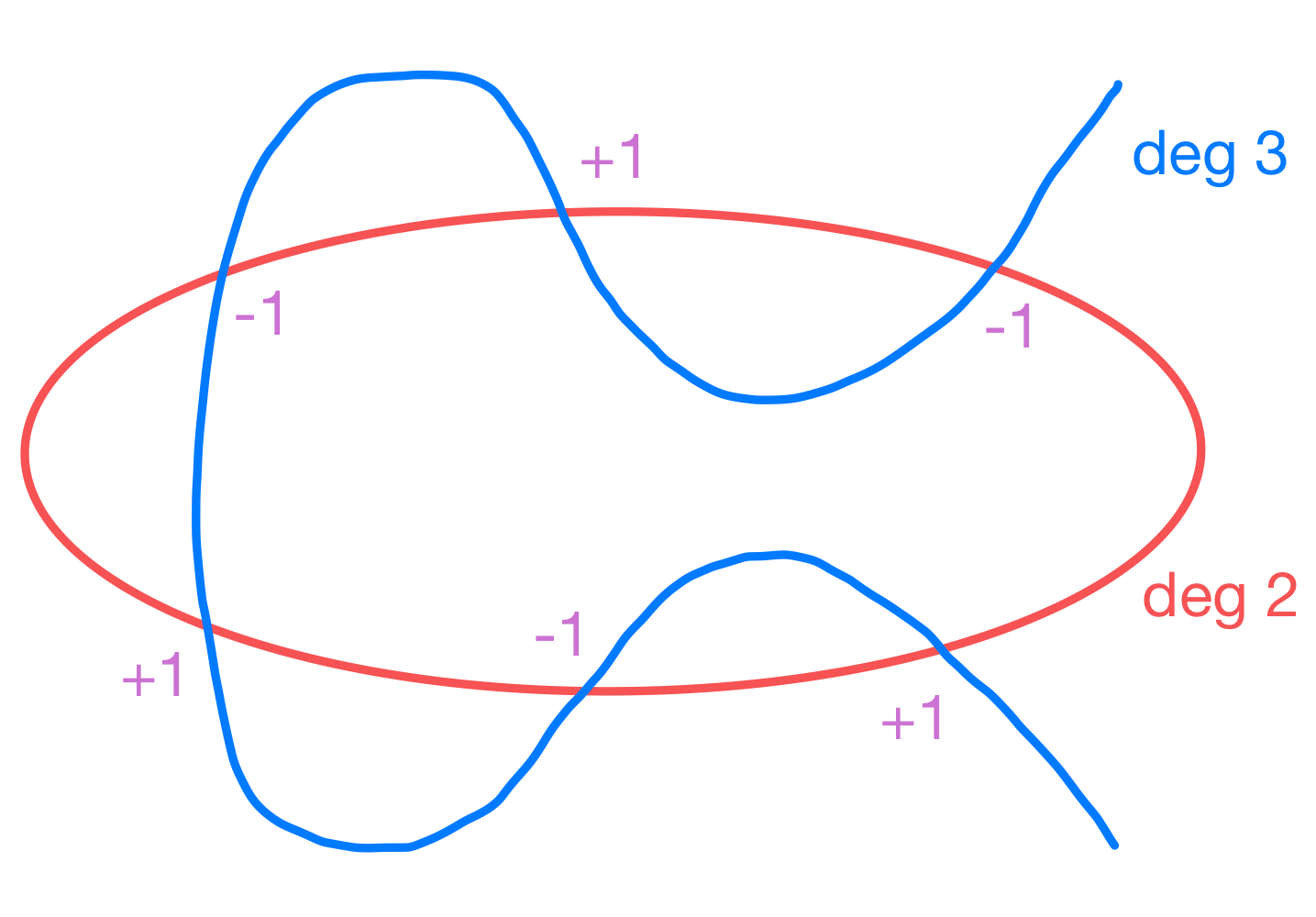}
\end{tabular}
\caption{Real Bézout for $n=1$ and $n=2$.}
\label{fig:realbezout}
\end{figure}

\end{ex}

\subsection{Enumerative Geometry over an arbitrary field $k$}
Next we want to answer the following questions. 
\begin{question}
    What about other (non-algebraically closed) fields $k$?
    Is there some theory that knows about the results in classical and real enumerative geometry but also has information about other fields?
\end{question}
To answer these questions we need to borrow a result from motivic homotopy theory, the analog of the degree \eqref{eq:degree} we used to define the local index.

\begin{thm}[Morel]
Let $k$ be a perfect field. Then there exists a well-defined degree map
    \[\deg^{\A^1}\colon [\nicefrac{\P^n_k}{\P^{n-1}_k}, \nicefrac{\P^n_k}{\P^{n-1}_k}]_{\A^1}\rightarrow \GW(k)\]
    Here $[-,-]_{\A^1}$ denotes the motivic/$\A^1$-homotopy classes and $\GW(k)$ the Grothendieck-Witt ring of $k$. 
\end{thm}
We will recall the definition of the Grothendieck-Witt ring in the next section and treat it as a black box for now. Now we simply copy the definitions from Definition \ref{def:relativeorientationandlocalindex} and adapt them to the algebraic/motivic setting. The following was first done in \cite{KassWickelgrenCubicSurface}.

\begin{definition}
%Let $k$ be a perfect field.
Let $p\colon V\rightarrow X$ be an algebraic vector bundle of rank $n$ over a smooth $n$-dimensional $k$-scheme $X$.
\begin{itemize}
     \item 
     Let $T_X$ be the tangent bundle of $X$.
     A \emph{relative orientation} of $p\colon V\rightarrow X$ consists of the data of a line bundle $\mathcal{L}\rightarrow X$ and an isomorphism of line bundles $\rho\colon \sheafHom(\det T_X,\det V)\cong \mathcal{L}^{\otimes 2}$. A vector bundle $p\colon V\rightarrow X$ is \emph{relatively orientable} if a relative orientation exists.
     \item Let $x\in X$ be a closed point and let $U$ be a Zariski open neighborhood of $x$. An étale map $\psi\colon U\rightarrow \A^n_k$ which induces an isomorphism of residue fields at $x$ is called \emph{Nisnevich coordinates}. By \cite[Lemma 19]{KassWickelgrenCubicSurface} Nisnevich coordinates always exist whenever $\dim X=n\ge 1$.
     \item Assume $p\colon V\rightarrow X$ is relatively oriented with relative orientation $\rho$. Let $x\in X$ be a closed point and let $\psi\colon U\rightarrow \A^n_k$ be Nisnevich coordinates of $x$ such that there is a local trivialization of $V\vert_U$. We say that Nisnevich coordinates around $x$ and a trivialization $V\vert_U\cong U\times \A^n_k$ are \emph{compatible} with $\rho$ if the distinguished section of $\sheafHom(\det T_X\vert_U,\det V\vert_U)$ taking the distinguished basis of $\det T_X\vert_U$ to $\det V\vert_U$ is sent to a square by $\rho$, that is the image of the distinguished section can be written as $l\otimes l$ where $l$ is a section of $\mathcal{L}$.
     \item Assume $p\colon V\rightarrow X$ is relatively oriented by $\rho$ and $\sigma\colon X\rightarrow V$ is a section. Let $x$ be an isolated zero of $\sigma$. Then locally around $x$ choose Nisnevich coordinates and a trivialization of $V$ compatible with $\rho$. 
     Nisnevich coordinates allow to express the section $\sigma$ locally around $x$ as a map $\sigma\colon \A^n_{\kappa(x)}\rightarrow \A^n_{\kappa(x)}$, where $\kappa(x)$ is the residue field at $x$, and
     \[\ind_x\sigma\coloneqq \Tr_{\kappa(x)/k}(\deg^{\A^1}_x\sigma)\]
     where $\deg^{\A^1}_x\sigma$ is the ``local $\A^1$-degree" and $\Tr_{\kappa(x)/k}$ is the ``trace map" which both will be defined in the next section
     
    % Then consider the composition 
    % \[U\xrightarrow{\sigma\vert_U} V\vert_U\cong U\times \A_k^n\xrightarrow{\operatorname{pr}_2}\A^n_k.\]
    %Nisnevich coordinates allow to identify $\nicefrac{U}{U\setminus \{x\}}$ with $\nicefrac{\A^n_k}{\A^n_k\setminus\{\psi(x)\}}$ in the $\A^1$-homotopy category. Furthermore, $\A^n_k$ naturally sits inside $\P^n_k$ and in the $\A^1$-homotopy category $\nicefrac{\A^n_k}{\A^n_k\setminus\{\psi(x)\}}\simeq \nicefrac{\P^n_k}{\P^n_k\setminus \{\psi(x)\}}$.
    %The \emph{local index} $\ind_x\sigma$ of $\sigma$ at $x$ is the $\A^1$-degree of 
    %\[\nicefrac{\P^n_k}{\P^{n-1}_{k}}\rightarrow \nicefrac{\P^n_k}{\P^n_k\setminus \{\psi(x)\}}\simeq\nicefrac{U}{U\setminus \{x\}}\xrightarrow{\overline{\sigma}}\nicefrac{\A^n_k}{\A^n_k\setminus \{0\}} \simeq\nicefrac{\P^n_k}{\P^{n-1}_k}. \]
\end{itemize}  
     
\end{definition}

%\begin{rmk}
%    We will define the local $\A^1$-degree of a map $f\colon \A^n_k\rightarrow \A^n_k$ at a zero in Definition \ref{def:localA1degree} and provide a formula for computing it in Proposition \ref{prop:localA1degree}.
%\end{rmk}

\begin{ex}
    For both $\P^n_k$ and $\Gr(2,4)$ for any closed point we can find a Zariski local neighborhood isormophic to $\A^n_{k}$. So here it is straightforward to write down Nisnevich coordinates and express a section $\sigma$ locally as $\sigma\colon \A^n_{\kappa(x)}\rightarrow \A^n_{\kappa(x)}$.
\end{ex}

The following theorem was first proved by Kass-Wickelgren \cite{KassWickelgrenCubicSurface} with some extra constraints and tailored for the example of lines on a smooth cubic surface, and then proved in general by Bachmann-Wickelgren \cite{BachmannWickelgren}.
\begin{thm}
\label{thm:A1PHthm}
Suppose $p\colon V\rightarrow X$ is an algebraic vector bundle of rank $n$ over a smooth proper $n$-dimensional $k$-scheme $X$ relatively oriented by $\rho$ and let $\sigma$ be a section with only isolated zeros. 
Then the \emph{Poincaré Hopf Euler number}
\[n^{\A^1}(V,\rho)\coloneqq \sum_{\text{zeros of }\sigma}\ind_x\sigma\]
is a well-defined element of $\GW(k)$. In particular, this means that $n^{\A^1}(V,\rho)$ is independent of the choice of section.
\end{thm}
\begin{rmk}
The Poincaré Hopf Euler number does depend on the choice of relative orientation. However, in practice one often drops $\rho$ in the notation, i.e. one writes $n^{\A^1}(V)=n^{\A^1}(V,\rho)$. 
    
\end{rmk}

\section{Formulas for the $\A^1$-degree}
In the second lecture, we begin by introducing the Grothendieck–Witt ring $\GW(k)$ of a field $k$. Recall that this is the target of the $\A^1$-degree map and serves as the natural home for our enumerative geometry invariants.
We then discuss several fundamental properties of $\GW(k)$, including its generators and relations, the trace map, and important structural maps such as the rank and signature. Along the way, we consider key examples, such as $\GW(\C)$ and $\GW(\R)$. We also introduce the Witt ring $\operatorname{W}(k)$. For a more detailed exposition of these topics, see \cite{Lam}.

Next, we define the local $\A^1$-degree and recall explicit formulas for its computation, drawing on the results of \cite{KassWickelgrenEKL,EisenbudLevine,Khimshiashvili,BrazeltonBurklundMcKeanMontoroOpie,BrazeltonMcKeanPauli}.
Finally, we return to our motivating examples, Bézout’s theorem and lines on a smooth cubic surface. The relevant references for these applications are \cite{KassWickelgrenCubicSurface,McKeanBezout}.

\subsection{The Grothendieck--Witt ring $\GW(k)$}

We start with the definition of the Grothendieck-Witt ring of a field $k$. For this
recall that a set $M$ equipped with a binary operation $\circ\colon M\times M\rightarrow M$ is a \emph{monoid} if 
\begin{itemize}
    \item for $a,b,c\in M$ you have $(a\circ b)\circ c=a\circ(b\circ c)$ (associativity),
    \item there exists $e\in S$ such that $e\circ a=a\circ e=a$ for all $a\in M$ (identity element).
\end{itemize}
A monoid $(M,\circ)$ is commutative, if the binary operation is commutative, i.e. $a\circ b=b\circ a$ for all $a,b\in M$.
A \emph{monoid homomorphism} $f\colon (M,\circ_M)\rightarrow (N,\circ_N)$ is a map $f\colon M\rightarrow N$ such that $f(a\circ_Mb)=f(a)\circ_Nf(b)$ for all $a,b\in M$ and $f(e_M)=e_N$ (these are the identity elements of the respective monoids).

If you also had inverses, this would be a group which is abelian if the monoid is commutative.
Groups are quite well understood. For example, finitely generated abelian groups are fully classified. Furthermore, doing computations in groups works much better than in monoids, since you have inverses. So in many situations you want to turn a commutative monoid into an abelian group, which can be done with the following construction.
\begin{definition}
        Let $(M,\circ)$ be a commutative monoid. Its \emph{Grothendieck group} or \emph{group completion} is an abelian group $K_0(M)$ together with a monoid homomorphism $i\colon M\rightarrow K_0(M)$ which satisfies the following universal property. For any abelian group $A$ and monoid homomorphism $f\colon M\rightarrow A$, there exists a unique group homomorphism $g\colon K_0(M)\rightarrow A$ such that 
    $f=g\circ i$.
    \[% https://tikzcd.yichuanshen.de/#N4Igdg9gJgpgziAXAbVABwnAlgFyxMJZABgBpiBdUkANwEMAbAVxiRAFkQBfU9TXfIRQBGclVqMWbAILdeIDNjwEiZYePrNWiEAGkA+sQAU7AJTdxMKAHN4RUADMAThAC2SMiBwQkAJmqaUjpYco4u7oie3kiiElpsDqEgzm5+1NGIsQxYYNogUBA4OFYgAZJ5ADoVMAAeWHA4cACEAATWFlxAA
\begin{tikzcd}
M \arrow[d, "i"] \arrow[r, "f"]         & A \\
K_0(M) \arrow[ru, "\exists! g", dotted,swap] &  
\end{tikzcd}\]
\end{definition}
\noindent
One can show that $K_0(M)$ exists by constructing it (see for example \cite[Section II.1]{Kbook}).

\medskip

For the whole section let $k$ be a field. 
A \emph{symmetric bilinear form} over $k$ is a bilinear map
\[b\colon V\times V\rightarrow k\]
where $V$ is a $k$-vector space such that $b(x,y)=b(y,x)$ for all $x,y\in V$. It is \emph{non-degenerate} if $V\rightarrow \Hom_k(V,k)$ given by $(v\mapsto (x\mapsto b(v,x))$ is an isomorphism.
Two non-degenerate symmetric bilinear forms $b_1\colon V_1\times V_1\rightarrow k$ and $b_2\colon V_2\times V_2\rightarrow k$ are \emph{isometric} if there is a linear isomorphism $\phi\colon V_1\rightarrow V_2$ such that $b_2(\phi(x),\phi(y))=b_1(x,y)$ for all $x,y\in V_1$. This defines an equivalence relation called \emph{isometry}.
Consider the set 
\[M\coloneqq \{\text{non-degenerate symmetric bilinear forms over $k$}\}/\text{isometry}\]
This can be given a monoid structure $(M,\oplus)$ with addition the direct sum: Let $b_1\colon V_1\times V_1\rightarrow k$ and $b_2\colon V_2\times V_2\rightarrow k$ be two non-degenerate symmetric bilinear forms, then $b_1\oplus b_2\colon (V_1\oplus V_2)\times (V_1\oplus V_2)\rightarrow k$ is also a non-degenerate symmetric bilinear form. You can easily check that this operation is still well defined when you pass to isometry classes.

\begin{rmk}
    When $\operatorname{char} k\neq 2$ a non-degenerate symmetric bilinear form $b\colon V\times V\rightarrow k$ defines a non-degenerate quadratic form by 
    \[q\colon V\rightarrow k,\; q(x)=b(x,x)\]
    and given a non-degenerate quadratic form $q\colon V\rightarrow k$ one gets a non-degenerate symmetric bilinear form
    \[b\colon V\times V\rightarrow k,\; b(x,y)=\frac{1}{2}(q(x+y)-q(x)-q(y)).\]
    So equivalently we can also think of elements of $\GW(k)$ as isometry classes of non-degenerate quadratic forms.
\end{rmk}

\begin{definition}
    The \emph{Grothendieck-Witt group $\GW(k)$ of $k$} is the group completion of the monoid of isometry classes of non-degenerate symmetric bilinear forms over $k$, that is
    \[\GW(k)=K_0(M,\oplus).\]
\end{definition}

We can define a second binary operation $\otimes$ on $M$ by taking the tensor product. This turns $\GW(k)$ into a ring.

There is a very nice presentation of $\GW(k)$ which is motivated by the following observation. When $\operatorname{char}k\neq 2$ any non-degenerate symmetric bilinear form $b\colon V\times V\rightarrow k$ over $k$ can be diagonalized, that is one can find a basis for $V$ such that the Gram matrix of $b$ is diagonal with entries $a_1,\ldots,a_n\in k^\times$ (here $n=\dim_kV$). In particular, the isometry class of $b$ equals the isometry class of the direct sum of the non-degenerate symmetric bilinear forms $k\times k\rightarrow k$ defined by $(x,y)\mapsto a_ixy$ for $i=1,\ldots n$. It follows that $\GW(k)$ is generated by the isometry classes of non-degenerate symmetric bilinear forms of the form $k\times k\rightarrow k$ defined by $(x,y)\mapsto a xy$ for $a\in k^\times$. We denote these generators by $\langle a\rangle$.
The following proposition also holds for $\operatorname{char}k=2$.
\begin{prop}
    $\GW(k)$ is generated by symbols $\langle a\rangle$ for $a\in k^\times$ subject to the relations
    \begin{enumerate}
        \item $\langle a\rangle =\langle ab^2\rangle$ for $a,b\in k^\times$,
        \item $\langle a\rangle+\langle b\rangle=\langle a+b\rangle+\langle ab(a+b)\rangle$ for $a,b,a+b\in k^\times$,
        \item $\langle a \rangle\langle b\rangle=\langle ab\rangle$ for $a,b\in k^\times$.
    \end{enumerate}
\end{prop}

\begin{exercise}
\label{exercise:hyperbolicform}
    Show that it holds that 
    \[\langle a\rangle +\langle -a\rangle =\langle 1\rangle+\langle-1\rangle\]
    for all $a\in k^\times$.
\end{exercise}

\begin{definition}
    The class $h\coloneqq \langle 1\rangle +\langle -1\rangle\in \GW(k)$ is called the \emph{hyperbolic form}.
\end{definition}

\begin{definition}
    The \emph{Witt ring} $\operatorname{W}(k)$ is defined by
    \[\operatorname{W}(k)\coloneqq \frac{\GW(k)}{(h)}=\frac{\GW(k)}{\Z\cdot h}\]
    that is we quotient by the ideal generated by $h$ which is indeed just all integer multiples $\Z\cdot h$ of $h$ by Exercise \ref{exercise:hyperbolicform}.
\end{definition}

Let $b\colon V\times V\rightarrow k$ be a non-degenerate symmetric bilinear form over $k$. Then the \emph{rank} of $b$ is defined to be $\operatorname{rk}(b)\coloneqq \dim_kV$. Clearly, isometric non-degenerate symmetric bilinear forms have the same rank and we get a monoid homomorphism $\operatorname{rk}\colon M\rightarrow \Z$ which extends uniquely to $\operatorname{rk}\colon \GW(k)\rightarrow \Z$ by the universal property of the Grothendieck group.

\begin{rmk}
\label{rmk:wittandrankdeterminegw}
    An element of $\GW(k)$ with $\operatorname{char}k\neq 2$ is completely determined by its rank and its image in $\operatorname{W}(k)$.
\end{rmk}

\begin{ex}
    Let $k$ be an algebraically closed field. Then the rank is an isomorphism $\operatorname{rk}\colon \GW(k)\xrightarrow{\cong} \Z$.
\end{ex}

Assume $k\subset \R$ and let $b\colon V\times V\rightarrow k$ be a non-degenerate symmetric bilinear form. We can again choose a basis of $V$ such that the Gram matrix has only diagonal entries $a_1,\ldots,a_n\in k^\times$ with $n=\operatorname{rk}(b)$. The signature $\operatorname{sgn}(b)$ of $b$ equals the number of positive $a_i$ minus the number of negative $a_i$. One can show that the signature also defines a monoid homomorphism $\operatorname{sgn}\colon M\rightarrow \Z$ and thus a group homomorphism $\operatorname{sgn}\colon \GW(k)\rightarrow \Z$.

\begin{ex}
    An element of $\GW(\R)$ is completely determined by its rank and its signature.
\end{ex}

Suppose $L/k$ is a finite separable field extension and let $b\colon V\times V\rightarrow L$ be a non-degenerate symmetric bilinear form over $L$. Then we can transform it into a non-degenerate symmetric bilinear form over $k$ by viewing $V$ as a $k$-vector space and composing $b$ with the field trace $\Tr_{L/k}$ from algebra
\[\Tr_{L/k}(b)\colon V\times V\xrightarrow{b} L\xrightarrow{\Tr_{L/k}}k\]
One checks that this respects isometry classes and thus defines a homomorphism of Grothendieck-Witt groups
\[\Tr_{L/k}\colon \GW(L)\rightarrow \GW(k).\]

\begin{exercise}
    Show that 
    \[\Tr_{\C/\R}(\langle 1\rangle)=h.\]
\end{exercise}

\subsection{The local $\A^1$-degree: formulas}
Now let's go back to the $\A^1$-degree map defined by Morel
\[\deg^{\A^1}\colon [\nicefrac{\P^n_k}{\P^{n-1}_k},\nicefrac{\P^n_k}{\P^{n-1}_k}]_{\A^1}\rightarrow \GW(k).\]

We blackbox motivic homotopy theory but still need to explain two things.
\begin{enumerate}
    \item What is this quotient $\nicefrac{\P^n_k}{\P^{n-1}_k}$?
    \item What is $[-,-]_{\A^1}$?
\end{enumerate}
Here is the intuition. In motivic homotopy theory we want to define the ``homotopy category" of smooth, separated schemes of finite type over $k$. In particular, one wants many of the constructions from homotopy theory to work here and one wants to have limits and colimits. This is achieved by ``formally adding" limits and colimits.
So if we have a morphism $U\rightarrow X$ of smooth, separated, finite type schemes over $k$, we can take the colimit of the diagram 
\[
\begin{tikzcd}
    U\arrow{r}\arrow{d}& X\\
   \operatorname{Spec} k&
\end{tikzcd}
\]
and denote it $\nicefrac{X}{U}$. This explains what $\nicefrac{\P^n_k}{\P^{n-1}_k}$ is.
We do not want to go through the construction of the motivic homotopy category here, since there are several good surveys on this (see for example \cite{WickelgrenWilliams}). We denote $[-,-]_{\A^1}$ the morphisms in this category which answers the second question.

Next we explain how the $\A^1$-degree relates to the degree map in algebraic topology. Assume $k\subset \R$. 
One can define complex and real realization functors that send a scheme to the complex respectively real manifold of its complex respectively real points. These extend to functors from the motivic homotopy category called \emph{complex} resp \emph{real realization} (see for example \cite[$\S3$]{WickelgrenWilliams}).
These realization functors are left adjoints and therefore commute with taking colimits. Consequently, taking the complex respectively real realization of the motivic $n$-sphere $\nicefrac{\P^n_k}{\P^{n-1}_k}$ yields
\[
\nicefrac{\P^n_k(\C)}{\P^{n-1}_k(\C)}\cong\nicefrac{\mathbb{CP}^n}{\mathbb{CP}^{n-1}}
\cong S^{2n}
\qquad\text{and}\qquad
\nicefrac{\P^n_k(\R)}{\P^{n-1}_k(\R)}\cong\nicefrac{\mathbb{RP}^n}{\mathbb{RP}^{n-1}}
\cong S^{n}.
\]
In fact one has the following commutative diagram relating the $\A^1$-degree with the degree from algebraic topology
\[% https://tikzcd.yichuanshen.de/#N4Igdg9gJgpgziAXAbVABwnAlgFyxMJZARgBoAGAXVJADcBDAGwFcYkRkAdTgQQD0wAem78w3ODBwBbLGGZwABN2DluAX1IiBw3gPGSZcuMtWc1lAPrAtxNSA3pMufIRQAmCtTpNW7ZAGUBUkCwSntSR2w8AiJyTxoGFjZEDkDgNzANNIzzcMjnGJQ44i9E3xTuAC08kAwolyIPEoSfZJAqmrqC1xJSZu8k9m4AcQB1AAoAawBKey8YKABzeCJQADMAJwgpJDiQHAgkDwHy9s4AYW4cGAAPHGAAWgxZHDg7BxBN7d2aA6QyE5tbgAJSut3uTwgLzeNS+O0QAL+iAALC1BhVOLBFrCtvC9kiAKwfOFHX6HRAAZjRp24WL41l0thx30QBLJSCpgKGnA2k2Z8LZ+3JqK5GLgi0I5jUQA
\begin{tikzcd}
{[S^{2n},S^{2n}]} \arrow[d, "\deg"] & {[\nicefrac{\P^n_k}{\P^{n-1}_k},\nicefrac{\P^n_k}{\P^{n-1}_k}]_{\A^1}} \arrow[l, "\C\text{-points}",swap] \arrow[r, "\R\text{-points}"] \arrow[d, "\deg^{\A^1}"] & {[S^n,S^n]} \arrow[d, "\deg"] \\
\Z                                         & \GW(k) \arrow[l, "\rk",swap] \arrow[r, "\sgn"]                                                                                         & \Z                           
\end{tikzcd}\]

Let's go back to algebraic topology for a bit. Let
\[f\colon S^n\rightarrow S^n\]
and let $y$ be a point in the target with finitely many preimages. For a preimage $x\in f^{-1}(y)$ we can choose a small ball $B(x)\subset S^n$ around $x$ such that $x$ is the only preimage in this ball, also choose a small ball $B(y)$ around $y$ in the target such that the boundary $\partial B(x)$ of $B(x)$ is mapped to the boundary $\partial B(y)$ of $B(y)$. Then the \emph{local degree} $\deg_xf$ at $x$ is the degree of the map
\[\overline{f}\colon S^n\simeq\nicefrac{B(x)}{\partial B(x)}\rightarrow \nicefrac{B(y)}{\partial B(y)}\simeq S^n.\]
Caution: When choosing the homotopy equivalences one needs to watch orientations.

The following result is well-known, see for example Hatcher \cite[Chapter 2.2]{Hatcher}. The sum of local degrees at the preimages of $y$ equals the degree of $f$
\[\deg f=\sum_{x\in f^{-1}(y)}\deg_xf.\]
When $y$ is a regular value then locally around all preimages $x$ the map $f$ is a homeomorphism and thus $\overline{f}$ is a homeomorphism and $\deg_xf\in \{\pm1\}$.
In particular, one has the following formula for $\deg_xf$ from differential topology.
Choose local oriented coordinates around $x$ and $y$. Then in these coordinates the local degree $\deg_xf$ at $x$ 
\begin{equation}
\label{eq:localdegreedifftopformula}
\deg_xf=\operatorname{sign}(\det \operatorname{Jac}f(x))
\end{equation}
is the sign of the determinant of the Jacobian of $f$ evaluated at $x$.

The goal for the rest of this section is to define the local $\A^1$-degree and write down an analogous formula for it.
This builds on work of Kass-Wickelgren \cite{KassWickelgrenEKL}.
Assume $U\subset \A^n_k\subset \P^n_k$ and $U$ is Zariski open in $\A^n_k$. Then in the motivic homotopy category $\nicefrac{U}{U\setminus \{x\}}\simeq\nicefrac{\P^n_k}{\P^n_k\setminus \{x\}}$ is an isomorphism.
\begin{definition}
\label{def:localA1degree}
    Let $x$ be an isolated zero of $f\colon \A^n_k\rightarrow \A^n_k$. Find a Zariski neighborhood $U$ of $x$ such that $f^{-1}(0)\cap U=\{x\}$. Then the \emph{local $\A^1$-degree of $f$ at $x$} is defined to be
    \[\deg^{\A^1}_xf\coloneqq \deg^{\A^1}\left(\nicefrac{\P^n_k}{\P^{n-1}_{k}}\rightarrow \nicefrac{\P^n_k}{\P^n_k\setminus \{x\}}\simeq\nicefrac{U}{U\setminus \{x\}}\xrightarrow{\overline{f}}\nicefrac{\A^n_k}{\A^n_k\setminus \{0\}} \simeq\nicefrac{\P^n_k}{\P^{n-1}_k}\right).\]
   % {\color{red} explain the $\A^1$-homotopy equivalences}
\end{definition}
The following proposition of Kass-Wickelgren generalizes the formula \eqref{eq:localdegreedifftopformula} from differential topology \cite{KassWickelgrenEKL}.
\begin{prop}
\label{prop:localA1degree}
Let $x$ be a preimage of $0$ under $f\colon \A^n_k\rightarrow \A^n_k$ and assume that the residue field $\kappa(x)$ is a separable field extension of $k$. Furthermore, assume that $\det \operatorname{Jac}f (x)\neq 0$. Then
\[\deg^{\A^1}_xf=\Tr_{\kappa(x)/k}(\langle \det \operatorname{Jac} f(x)\rangle)\in \GW(k).\]
\end{prop}

\begin{ex}
    This agrees with the formula for the local degree in topology from differential topology:
    If $k=\R$ and $\kappa(x)=\R$ then $\deg^{\A^1}_xf=\langle \operatorname{sign}(\det \operatorname{Jac} f(x))\rangle$. 
\end{ex}

So what if $\det \operatorname{Jac}f (x)= 0$ or $\kappa(x)$ is not a separable field extension of $k$?
There are more general formulas. Let's first consider the case that $\det \operatorname{Jac}f (x)= 0$. For this case, Eisenbud-Levine \cite{EisenbudLevine} and independently Khimshiashvili \cite{Khimshiashvili} found a formula for the local degree in topology, that is over $k=\R$. It was shown by Kass-Wickelgren \cite{KassWickelgrenEKL} that the same formula works for the local $\A^1$-degree when $\kappa(x)=k$ and then later generalized by Brazelton-Burklund-McKean-Montoro-Opie for when $\kappa(x)/k$ is separable. Finally, in \cite{BrazeltonMcKeanPauli} it is shown that the ``multivariate Bézoutian'' gives a formula for the local $\A^1$-degree that always works, in particular even when $\kappa(x)/k$ is not separable.
All of this builds on work of Scheja-Storch \cite{SchejaStorch}, who found a canonical way to assign a non-degenerate symmetric bilinear form to a complete intersection and \cite{KassWickelgrenEKL,BachmannWickelgren} showed that the class of this non-degenerate symmetric bilinear form agrees with the local $\A^1$-degree.
Note that there is now also a Macaulay2 package that can compute local $\A^1$-degrees \cite{M2package}, see also \cite{pauliM2}.

The following is the general formula for the (local) $\A^1$-degree as in \cite{BrazeltonMcKeanPauli} building on \cite{SchejaStorch,KassWickelgrenEKL,BachmannWickelgren}.
Let $f=(f_1,\ldots,f_n)\colon \A^n_k\rightarrow \A^n_k$ be a map with only finitely many zeros. In other words
\[Q\coloneqq \frac{k[x_1,\ldots,x_n]}{(f_1,\ldots,f_n)}\]
is a finite dimensional $k$-vector space. Let $m=\dim_kQ$. Choose a basis $a_1,\ldots,a_m$ for $Q$ as a $k$-vector space. 
Set
\[\Delta_{ij}\coloneqq \frac{f_i(Y_1,\ldots,Y_{j-1},X_j,\ldots,X_n)-f_i(Y_1,\ldots,Y_j,X_{j+1},\ldots,X_n)}{X_j-Y_j}\]
for $i,j\in \{1,\ldots,n\}$.
Then the \emph{multivariant Bézoutian} $\operatorname{Bez}(f_1,\ldots,f_n)$ of $f_1,\ldots,f_n$ is
the image of $\det (\Delta_{ij})_{i,j=1}^n$ in $\frac{k[X,Y]}{(f(X),f(Y))}\cong Q\otimes_kQ$.
A basis for $Q\otimes_kQ$ is given by $a_i\otimes a_j$ for $i,j=1,\ldots, m$ and we can write
\[\operatorname{Bez}(f_1,\ldots,f_n)=\sum_{i,j=1}^mB_{ij}a_i\otimes a_j\]
with $B_{ij}\in k$ uniquely. In fact, the matrix $(B_{ij})_{i,j=1}^m$ is symmetric and has non-vanishing determinant. Hence, it defines a non-degenerate symmetric bilinear form and the class of the form equals
\[\deg^{\A^1}f=\sum_{\text{$x$ zero of $f$}}\deg_x^{\A^1}f.\]
One can also compute the local $\A^1$-degrees $\deg_x^{\A^1}f$ using the multivariate Bézoutian. For this one replaces $Q$ above by the local ring $Q_\mathfrak{m}$ where $\mathfrak{m}$ is the maximal ideal corresponding to the zero $x$ of $f$.

\subsection{Back to Bézout's theorem and lines on a smooth cubic surface}
Now let's return to our two examples in enumerative geometry, namely Bézout and lines on a smooth cubic surface.
\begin{ex}[Bézout continued]
\label{ex:bezoutoveranyfield}
Let $p\colon V=\glob(d_1)\oplus\ldots\oplus \glob(d_n)\rightarrow \P^n_k$. Then
    \[\sheafHom(\det T_{\P^n_k},\det V)\cong \omega_{\P^n_k}\otimes \glob(d_1+\ldots+d_n)\cong \glob(-n-1+\sum d_i)\]
    and thus $p\colon V\rightarrow \P^n_k$ is relatively orientable if and only if $-n-1+\sum d_i$ is even. Note that this agrees with the relative orientability condition in the real case in Example \ref{ex:bezoutreal}. For example, when $n=1$ then $d_1$ must be even, and for $n=2$ one $d_i$ must be even and the other has to be odd.
    Assume that all intersections of the hypersurfaces $H_i=V(F_i)$ lie in $U=\{x_0\neq 0\}$. Let $f_i(x_1,\ldots,x_n)=F_i(1,x_1,\ldots,x_n)$.
     In \cite{McKeanBezout} McKean shows that if the hypersurfaces meet transversely at all intersection points, then at $x\in H_1\cap\ldots\cap H_n$ the local index 
    \[\ind_x\sigma_{F_1,\ldots,F_n}=\Tr_{\kappa(x)/k}(\langle\det \operatorname{Jac}(f_1,\ldots,f_n)(x)\rangle)\]
    and thus
    \[n^{\A^1}(V)=\sum_{x\in H_1\cap\ldots\cap H_n}\Tr_{\kappa(x)/k}(\langle\det \operatorname{Jac}(f_1,\ldots,f_n)(x)\rangle). \]
    McKean also shows that
    \[n^{\A^1}(V)=\frac{d_1\cdot\ldots\cdot d_n}{2}h.\]
    Let's compute the rank and the signature of this
    \[\rk\left(\frac{d_1\cdot\ldots\cdot d_n}{2}h\right)=d_1\cdot\ldots\cdot d_n\]
    which is the answer over $k=\C$ and
    \[\sgn\left(\frac{d_1\cdot\ldots\cdot d_n}{2}h\right)=0\]
    which is the real signed count of intersection points.
\end{ex}

\begin{ex}[Lines on a cubic surface continued]
    Kass-Wickelgren \cite{KassWickelgrenCubicSurface} show that $p\colon V=\Sym^3\mathcal{S}^*\rightarrow \Gr(2,4)$ is relatively orientable and that 
    \[n^{\A^1}(V)=\sum_{\ell\subset \{f=0\}}\ind_{[\ell]}\sigma_f =15\langle1\rangle+12\langle-1\rangle.\]
    Here, $f$ can be any degree $3$ homogeneous polynomial in $4$ variables which defines a smooth cubic surface $\{f=0\}\subset \mathbb{P}^3_k$, the sum above is independent of the choice of smooth cubic surface.
    Again this also yields the complex and real counts when taking the rank and the signature:
    \[\rk(15\langle1\rangle+12\langle-1\rangle)=27, \;\sgn(15\langle1\rangle+12\langle-1\rangle)=3. \]
    Kass-Wickelgren also provide an intrinsic way to assign a class $\Type(\ell)\in \GW(\kappa([\ell]))$ to a line $\ell\subset \{f=0\}$ on the cubic surface such that 
    \[\ind_{[\ell]}\sigma_f=\Tr_{\kappa([\ell])/k}(\Type(\ell))\]
    which generalizes Segre's work over $\R$ (see Example \ref{ex:reallines}).
    Namely, assume that $\ell$ is defined over $k$ and consider
    \begin{align*}
        G\colon \ell\cong \P^1_k&\rightarrow \text{$2$-planes in $\P^3_k$ containing $\ell$}\cong \P^1_k\\
        p&\mapsto T_pY
    \end{align*}
    where $Y=\{f=0\}\subset \P^3_k$ is the cubic surface.
    This turns out to be a degree $2$ map and thus there is a non-trivial involution $i\colon \ell\rightarrow \ell$ such that $G\circ i=G$ that is $i$ sends a point to the other point with the same tangent space. The fixed points of this involution are defined over $k(\sqrt{\alpha})$ for some $\alpha\in k^\times$. Kass-Wickelgren show that $\Type(\ell)=\langle\alpha\rangle\in \GW(k)$.
    This in fact agrees with Segre's sign when $k=\R$, on just has to put brackets $\langle \rangle$ around Segre's signs.
    %. In fact, when $k=\R$, $\alpha$ can be chosen to be either $+1$ or $-1$. Fixed points of the involution correspond to where the frame in Example \ref{ex:reallines} makes a turn. If $\alpha=+1$ then this is a real place and one comes back to where one has started if $\alpha=-1$ the frame does a full turn.

\end{ex}

\begin{rmk}
    In general it is a very hard problem to find an intrinsic description of the local index. A very nice survey of what is known is given in \cite{McKeanApollonius}. There are also many open problems listed. The definition of the type of a line on a smooth cubic surface can be generalized to lines on degree $2n-3$ hypersurfaces in $\mathbb{P}^n$ by \cite{pauliquintic} and \cite{EMP}.
\end{rmk}

\section{Plane tropical curves}

In the third lecture, we shift gears and explore the world of \emph{tropical geometry}.  
Tropical geometry studies \emph{piecewise-linear} objects that play a role similar to algebraic varieties in classical algebraic geometry.  
Just as in algebraic geometry, where we work with zero sets of polynomials, here we study objects defined by the \emph{tropical zero locus} of \emph{tropical polynomials}.

There is a process called \emph{tropicalization} that transforms objects from algebraic geometry into their tropical counterparts.  
We will review this process in detail, which requires introducing the field of Puiseux series $\Puiseux{k}$.

In this lecture, we will focus on \emph{plane tropical curves}, since our goal is to prove a tropical version of Bézout's theorem for curves.  
The ideas, however, extend naturally to higher-dimensional tropical varieties.  
For further reading and nice introductions to tropical geometry, we refer to \cite{BrugalleShawIntro,tropical_geometry,MacLaganSturmfels}.

We will conclude the lecture by proving a tropical version of Bézout's theorem for curves.  
In the tropical setting, this turns into a beautiful and purely combinatorial argument.

\subsection{Tropical geometry}
%Often one can translate problems in algebraic geometry to problems in tropical geometry where they can be solved using combinatorics. The great breakthrough in the use of tropical geometry in enumerative geometry was Milkalkin's tropical correspondence theorem, which allows one to translate the count of plane algebraic curves through a given number of points into the count of tropical curves through the same number of points \cite{Mikhalkin} and which we will return to in the last lecture. 

First we explain how to go from algebraic plane curves to tropical plane curves. To do this, we consider algebraic curves over the field of \emph{Puiseux series}, which we define next.
\begin{definition}
    Let $k$ be a field. Then the field of Puiseux series over $k$ is 
    \begin{align*}&\Puiseux{k}\coloneqq \bigcup_{n\ge 0}k((t^{\frac{1}{n}}))\\
    =&\{a_0t^{q_0}+a_1t^{q_1}+\ldots:\;a_i\in k,q_i\in \mathbb{Q}\text{, }q_0<q_1<\ldots,\text{ s.t. $\exists n\in \Z_{>0}$ s.t. $q_i\cdot n\in \Z$ for all $i$}\}.\end{align*}
\end{definition}
\begin{ex}
    When $\operatorname{char} k=0$, then $\overline{k((t))}=\Puiseux{\overline{k}}$.
    %In positive characteristic $p$ this does not suffice: 
    %\[t^{-\frac{1}{p}}+t^{-\frac{1}{p^2}} +t^{-\frac{1}{p^3}} +\ldots\]
    %is algebraic over $k((t))$ as it is a zero of the Artin-Schreier polynomial $x^p-x-t^{-1}$ but not in $\Puiseux{k}$. For more details see \cite{Puiseuxpositivecharacteristic}.
\end{ex}
\begin{exercise}
    Suppose $k$ has characteristic $p$ and is algebraically closed. Consider the following Artin-Schreier polynomial 
    $z^p-z-t^{-1}\in \Puiseux{k}[z]$. Show that 
    \[t^{-\frac{1}{p}}+t^{-\frac{1}{p^2}} +t^{-\frac{1}{p^3}} +\ldots\]
    is a zero of this polynomial and conclude that $\Puiseux{k}$ is not algebraically closed.
\end{exercise}

There is a \emph{valuation} on $\Puiseux{k}$
\begin{equation}\label{eq:valuation}\val\colon \Puiseux{k}\rightarrow \Q\cup\{\infty\}\end{equation}
defined by $\val(a_0t^{q_0}+a_1t^{q_1}+\ldots)=q_0$ and $\val(0)=\infty$. This satisfies
\begin{enumerate}
    \item $\val(a(t))=\infty$ $\Leftrightarrow$ $a(t)=0$,
    \item $\val(a(t)+b(t))\ge\min\{\val(a(t)),\val(b(t))\}$,
    \item $\val(a(t)\cdot b(t))=\val(a(t))+\val(b(t))$.
\end{enumerate}

Let's first assume that $k$ is of characteristic $0$ and algebraically closed. Then $\Puiseux{k}$ is also algebraically closed and of characteristic $0$. Let's try to find points in the zero locus of a polynomial $F\in \Puiseux{k}[z_1,z_2]$.

    We start with the simplest case that is that $\deg F=1$ and we can write 
    \[F(z_1,z_2)=\ta(t)+\tb(t)z_1+\tc(t)z_2\]
    with 
    \begin{align*}
        \ta(t)=at^{\val(\ta)}+h.o.t\in \Puiseux{k}\\
        \tb(t)=bt^{\val(\tb)}+h.o.t.\in \Puiseux{k}\\
        \tc(t)=ct^{\val(\tc)}+h.o.t.\in \Puiseux{k}
    \end{align*}
    where $h.o.t.$ is short for higher order terms in $t$. Next we want to describe the zeros $(\tp_1(t),\tp_2(t))\in \Puiseux{k}^2$
    with
    \begin{align*}
        \tp_1(t)=p_1t^{\val(\tp_1)}+h.o.t.\\
        \tp_2(t)=p_2t^{\val(\tp_2)}+h.o.t.
    \end{align*}
    In order for $$F(\tp_1(t),\tp_2(t))=(at^{\val(\ta)}+h.o.t.)+(bp_1t^{\val(\tb)+\val(\tp_1)}+h.o.t.)+(cp_2t^{\val(\tc)+\val(\tp_2)}+h.o.t.)=0$$ 
    it is necessary that the minimum of the three exponents of the three summands
    \[\{\val(\ta),\val(\tb)+\val(\tp_1),\val(\tc)+\val(\tp_2)\}\]
    is attained at least twice or equivalently
    the maximum in
    \[\{-\val(\ta),-\val(\tb)-\val(\tp_1),-\val(\tc)-\val(\tp_2)\}\]
    is attained at least twice. Only then one gets a linear equation in $p_1$ and $p_2$. 
\begin{rmk}
    In tropical geometry one typically works with either the min-plus or the max-plus convention. These two frameworks are equivalent up to a change of sign, and the distinction arises precisely from the sign choices made in the two equations shown above.
    In these lecture notes we use the max-plus convention.
\end{rmk}
    \begin{ex}
\label{ex:tropicalline}
    Let's look at an example and assume $\val(\ta)=2$, $\val(\tb)=-\frac{1}{2}$ and $\val(\tc)=0$
    and plot for which $x=-\val(\tp_1)$ and $y=-\val(\tp_2)$ this can be solved in Figure \ref{fig:firsttropicaline}.
    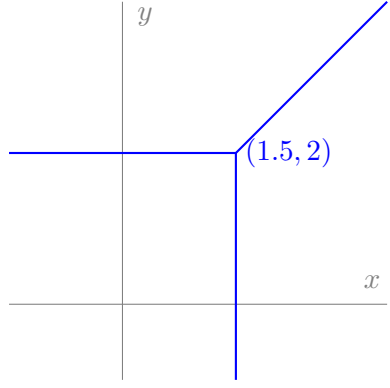
\begin{figure}
\begin{tikzpicture}[scale=1]
\draw[gray] (-3,-2)--(2,-2);
\draw[gray] (-1.5,2)--(-1.5,-3);
\node[gray] at (1.8,-1.7) {$x$};
\node[gray] at (-1.2,1.8) {$y$};
\draw[blue,thick] (-3,0)--(0,0)--(0,-3);
\draw[blue,thick] (0,0)--(2,2);
\node[blue] at (0.7,0) {$(1.5,2)$};
\end{tikzpicture}
\caption{$(x,y)\in \mathbb{Q}^2$ for which $F(\tp_1,\tp_2)=0$ can be solved with $x=-\val(\tp_1)$ and $y=-\val(\tp_2)$ in Example \ref{ex:tropicalline}.}
\label{fig:firsttropicaline}
\end{figure}
    This is actually our first example of a tropical plane curve.
\end{ex}

Now let's do this for polynomials $F(z_1,z_2)\in \Puiseux{k}[z_1,z_2]$ of higher degree. Let
 $F(z_1,z_2)=\sum \ta_{ij}(t)z_1^iz_2^j\in \Puiseux{k}[z_1,z_2]$ 
    of degree $d\ge 1$ with $\ta_{ij}(t)=a_{ij}t^{\val(\ta_{ij})}+h.o.t.\in \Puiseux{k}$.
    Then $F(z_1,z_2)$ can only have a zero $(\tp_1,\tp_2)\in \Puiseux{k}^2$
 if
    \begin{equation}
    \label{eq:maxattainedtwice}\max_{ij}\{i\cdot x+j\cdot y-\val(\ta_{ij})\}\end{equation}
    with $x=-\val(\tp_1)$ and $y=-\val(\tp_2)$
    is attained at least twice.

As before, set $x=-\val(\tp_1)$ and $y=-\val(\tp_2)$ for $\tp_1(t)=p_1t^{\val(\tp_1)}+h.o.t.\in \Puiseux{k}$ and $\tp_2(t)=p_2t^{\val(\tp_2)}+h.o.t.\in \Puiseux{k}$. Then
it follows from the properties of the valuation (1)-(3) listed above that
\begin{enumerate}
    \item $x=-\infty$ $\Leftrightarrow$ $\tp_1(t)=0$ (and the same for $y$ and $\tp_2$)
    \item $-\val(\tp_1(t)+\tp_2(t))\le \max\{x,y\}$
    \item $-\val(\tp_1(t)\cdot \tp_2(t))=x+y$
\end{enumerate}
This motivates the following definition of the tropical semi-ring.
\begin{definition}
    The \emph{tropical semi-ring} $(\mathbb{T},\oplus,\odot)$ is defined by
    \begin{itemize}
        \item $\mathbb{T}=\R\cup\{-\infty\}$
        \item $x\oplus y=\max(x,y)$
        \item $x\odot y=x+y$
    \end{itemize}
\end{definition}
This is not a ring/field since there are no inverses to tropical addition $\oplus$. However, $-\infty$ is the neutral element for tropical addition $\oplus$.

\begin{definition}
    A \emph{tropical polynomial} in two variables $x,y$ is of the form
    \[f=\bigoplus b_{ij}\odot x^{\odot i}y^{\odot j}=\max\{b_{ij}+ix+jy\}\]
    where $x^{\odot i}=x\odot\ldots\odot x$ with $i$ tropical factors and $y^{\odot j}=y\odot\ldots \odot y$ with $j$ tropical factors.
    with only finitely many $b_{ij}$ not equal to $-\infty$, that is the neutral element for tropical addition. 
    
    The \emph{tropical vanishing locus} of a tropical polynomial $f$ is given by
    \[V^{\trop }(f)=\{p\in\R^2:\text{the maximum is attained at least twice at $p$}\}.\]
\end{definition}

This might look familiar: Recall that given a polonomial $F(z_1,z_2)\in \Puiseux{k}[z_1,z_2]$, this describes a necessary condition for a zero to exist in terms of the valuation \eqref{eq:valuation}. In fact, one can turn a polynomial
\[F(z_1,z_2)=\sum \ta_{ij}(t)z_1^i z_2^j \in \Puiseux{k}[z_1,z_2]\]
into a tropical polynomial
\[F^{\trop}(x,y)=\bigoplus -\val(\ta_{ij}(t))\odot x^{\odot i}y^{\odot j}.\]
$F^{\trop}$ is called the \emph{tropicalization} of $F$.
%We have seen above that it is necessary for $(\tp_1(t),\tp_2(t))\in \Puiseux{k}^2$ to be in the vanishing locus of $F$, that $(-\val(\tp_1(t),-\val(\tp_2(t))))\in V^{\trop}(F^{\trop})$. 
The following theorem implies that the necessary condition for $(\tp_1(t),\tp_2(t))\in \Puiseux{k}^2$ to be in the vanishing locus of $F$ is actually sufficient.
\begin{thm}[Kapranov \cite{Kapranovstheorem}]
Let $k$ be an algebraically closed field of characteristic $0$. Then
$V^{\trop}(F^{\trop})$ equals the closure of 
\[\{(-\val(\tp_1(t)),-\val(\tp_2(t)))\in \mathbb{Q}^2: F(\tp_1(t),\tp_2(t))=0\}\]
in $\R^2$.
\end{thm}
We are ready to define plane tropical curves.

\begin{definition}
    A \emph{plane tropical curve} $\Gamma$ is an embedded weighted graph in $\R^2$ given by the tropical vanishing locus of a tropical polynomial $f$. The weights on the edges are defined as follows. 
    Let $e$ be an edge, then its weight $w(e)$ is
    \[w(e)\coloneqq \max\{\gcd(\vert i-k\vert,\vert j-l\vert):\; \{(i,j),(k,l)\}\in M_e\}\]
    with
    $M_e=\{\{(i,j),(k,l)\} :\; (i,j),(k,l)\in \Z^2,\; (i,j)\neq (k,l)\; \text{s.\ t.\ } \forall (p,q)\in e, f(p,q)=a_{ij}+ip+jq=a_{kl}+kp+lq\}$
    that is $M_e$ ranges over pairs of tuples where the maximum is attained.
\end{definition}
One only labels edges of weight $>1$. When there is no label, it means that the weight of this edge is $1$ (see Figure \ref{fig:extropcurves}).

\begin{rmk}
The definition of the weights of the edges might seem a little confusing and unnecessary at first. However, this is chosen such that the weight an edge equals the lattice length of the dual edge in the \emph{dual subdivision} (see paragraph below Definition \ref{def: Newton polygon}) and such that the \emph{balancing condition} (see Exercise \ref{exercise:balancing}) is satisfied.
    
\end{rmk}
\begin{definition}
    Let $C$ be a plane algebraic curve over $\Puiseux{k}$ defined by $F\in \Puiseux{k}[z_1,z_2]$ and let $\Gamma$ be the tropical curve defined by $F^{\trop}$. We say $C$ \emph{tropicalizes} to $\Gamma$. We also say a point $(\tp_1(t),\tp_2(t))\in \Puiseux{k}^2$ \emph{tropicalizes} to $(-\val(\tp_1(t)),-\val(\tp_2(t)))\in \R^2$. 

\end{definition}
    Note that points in $C=V(F)$ tropicalize to points in $\Gamma=V^{\trop}(F^{\trop})$ by Kapranov's theorem.

\begin{figure}
\begin{tabular}{ccc}
\begin{tikzpicture}
    % Vertex
    %\draw[step=0.5, gray!30] (-3,0) grid (1,4);
    
    %\draw[step=1, gray!30] (-1,-1) grid (3,3);
    \draw[-, thick, blue] (-1,0) -- (0,0) node[midway, right] {};
    \draw[-, thick, blue] (0,-1) -- (0,0) node[midway, right] {};
    \draw[-, thick, blue] (0,0) -- (1,1) node[midway, right] {};
    \draw[-, thick, blue] (1,1) -- (2,1) node[midway, right] {};
    \draw[-, thick, blue] (1,1) -- (1,2) node[midway, right] {};
    \draw[-, thick, blue] (2,1) -- (2,-1) node[midway, right] {};
    \draw[-, thick, blue] (1,2) -- (-1,2) node[midway, right] {};
    \draw[-, thick, blue] (2,1) -- (3,2) node[midway, right] {};
    \draw[-, thick, blue] (1,2) -- (2,3) node[midway, right] {};
    
   \node[below left, scale=0.8, red] at (-0.5, -0.5) {$0$};
   \node[below left, scale=0.8, red] at (1, 0) {$x$};
   \node[below left, scale=0.8, red] at (3, 0) {$2x-2$};
   \node[below left, scale=0.8, red] at (0, 1) {$y$};
   \node[below left, scale=0.8, red] at (1, 3) {$2y-2$};
   \node[below left, scale=0.8, red] at (3, 2.5) {$x+y-1$};
   %\node[below left, scale=0.7, black] at (0, 0) {$(0,0)$};
   % \node[below left, scale=0.8, red] at (1, 1) {$2+x$ is max};
   % \node[below left, scale=0.8, red] at (-1, 3) {$y-1$ is max};
\end{tikzpicture} &
 \begin{tikzpicture}[scale=0.8]
            % \draw[step=1, gray!30] (-2,-2) grid (4,3);
    \draw[-, thick, blue] (-2,0) -- (0,0)--(0,-2) node[midway, right] {};
    \draw[-, thick, blue] (0,0) -- (2,1)--(2,-2) node[midway, right] {};
    \draw[-, thick, blue] (2,1) -- (4,3) node[midway, right] {};

\node[below left, scale=0.8, blue] at (-1, 0.5) {$2$};
   \node[below left, scale=0.8, blue] at (3, 2.5) {$2$};

   \node[below left, scale=0.8, red] at (-0.5, -0.5) {$0$};
   \node[below left, scale=0.8, red] at (1, 0) {$x$};
   \node[below left, scale=0.8, red] at (4, 0) {$2x-2$};
   \node[below left, scale=0.8, red] at (0, 2) {$2y$};

%\node[below left, scale=0.7, black] at (0, 0) {$(0,0)$};
        \end{tikzpicture}
        &
        \begin{tikzpicture}[scale=0.7]
             
    \draw[-, thick, blue] (-1,0) -- (0,0)--(0,-1) node[midway, right] {};
    \draw[-, thick, blue] (0,0) -- (1,1)--(2,0)--(2,-1) node[midway, right] {};
    \draw[-, thick, blue] (2,0) -- (3,0)--(3,-1) node[midway, right] {};
    \draw[-, thick, blue] (3,0)--(4,1)node[midway, right] {};
    \draw[-, thick, blue] (1,1)--(1,2)--(0,3)--(-1,3)node[midway, right] {};
    \draw[-, thick, blue] (1,2)--(2.5,3.5)node[midway, right] {};
    \draw[-, thick, blue] (0,3)--(0,4)--(-1,4)node[midway, right] {};
    \draw[-, thick, blue] (0,4)--(1,5)node[midway, right] {};

\node[below left, scale=0.8, blue] at (1, 1.7) {$2$};

        \end{tikzpicture}

\end{tabular}
\caption{Examples of tropical curves.}
\label{fig:extropcurves}
\end{figure}
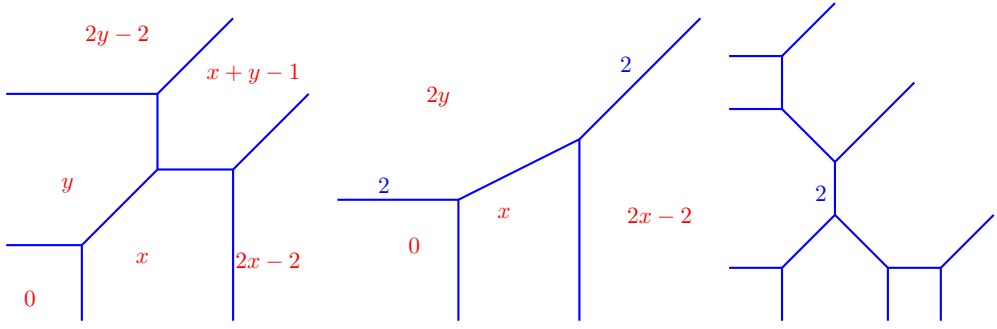

\begin{ex}
Consider the following tropical polynomial
      \[f=0\oplus x\oplus y\oplus(-1)\odot x\odot y\oplus(-2)\odot x^{\odot2}\oplus(-2)\odot y^{\odot2}=\max\{0,x,y,x+y-1,2x-2,2y-2\}.\]
      The left picture in Figure \ref{fig:extropcurves} shows the tropical curve defined by $f$. In the figure it is also indicated in red which piecewise linear functions out of $\{0,x,y,x+y-1,2x-2,2y-2\}$ is maximal in which segment.

For this curve all edges are of weight one.
\end{ex}

\begin{ex}This example is inspired by \cite[Figure 3]{BrugalleShawIntro}. It is the tropical curve defined by the tropical polynomial
\[f=0\oplus x\oplus y\oplus y^{\odot 2}\oplus(-2)\odot x^{\odot2}\]
and shown in the middle in Figure \ref{fig:extropcurves}. Here, we have edges of weight bigger than $1$. For example take the edge of weight $2$ on the left, that is the line $y=0$. Here, the maximum is attained by $0$, $y$ and $2y$. The corresponding tropical monomials are $0$, $y$ and $y^{\odot2}$. So for this edge, let's call it $e$, we have  $M_e=\{\{(0,0),(0,1)\},\{(0,1), (0,2)\},\{(0,0),(0,2)\}\}$ and
thus $w(e)=\max\{\gcd(\vert0-0\vert,\vert1-0\vert),\gcd(\vert0-0\vert,\vert2-0\vert),\gcd(\vert0-0\vert,\vert2-1\vert)\}=2$.

For the right edge, let's call it $e'$, the maximum is attained by $2x-2$ and $2y$ and we have $M_{e'}=\{\{(2,0),(0,2)\}\}$ and thus $w(e')=\gcd(2,2)=2$.
\end{ex}
\begin{ex}
    The right picture in Figure \ref{fig:extropcurves} is a tropical curve defined by a degree $3$ tropical polynomial.
    
\end{ex}

\begin{definition}
\label{def: Newton polygon}
    Let $\Gamma $ be a plane tropical curve defined by a tropical polynomial $f=\bigoplus b_{ij}\odot x^{\odot i}\odot y^{\odot j}$. The \emph{Newton polygon} $\operatorname{NP}(f)$ of $f$ is 
    the convex hull
    \[\operatorname{Conv}\{(i,j):\; b_{ij}\neq -\infty\}\subset \R^2\]
\end{definition}

The Newton polygon is a lattice polygon in $\R^2$ meaning that the vertices are in $\Z^2$.
To a tropical polynomial $f=\bigoplus b_{ij}\odot x^{\odot i}\odot y^{\odot j}$ one can assign a lattice subdivision of its Newton polygon called the \emph{dual subdivision} $\operatorname{DS}(f)$ in the following way. One projects the edges of the upper faces of 
\[\operatorname{Conv}(\{(i,j,b_{ij}): b_{ij}\neq-\infty\})\subset \R^3\]
to $\R^2$ via the projection to the first two coordinates.
%{\color{red} Todo: add an example}

One can show that there is the following one-to-one correspondence
\begin{center}
\begin{tabular}{|c|c|}

\hline
\textbf{Tropical curve $\Gamma$ defined by $f$} & \textbf{Dual subdivision $\operatorname{DS}(f)$} \\
\hline
vertex & connected component of $\operatorname{NP}(f)\setminus \operatorname{DS}(f)$ \\

edge of weight $w$ & edge of lattice length $w$ \\

connected component of $\R^2\setminus \Gamma$  & vertex \\
\hline

\end{tabular}
\end{center}
Moreover dual edges are orthogonal and inclusions are inverted.

\begin{ex}
    Figure \ref{fig:dualsubdivisions} shows the dual subdivisions of the tropical curves in Figure \ref{fig:extropcurves}.

\begin{figure}
\begin{tabular}{ccc}
\begin{tikzpicture}
\filldraw[fill=gray!10!white,thick,draw=none] (0,0)--(2,0)--(0,2)--cycle;
    \draw[step=1, gray!30] (-1,-1) grid (3,3);
    \draw[-, thick, blue] (0,0)--(2,0)--(0,2)--(0,0)node[midway, right] {};
    \draw[-, thick, blue] (0,1)--(1,1)--(1,0)--(0,1)node[midway, right] {};
    
\end{tikzpicture}
&
\begin{tikzpicture}
\filldraw[fill=gray!10!white,thick,draw=none] (0,0)--(2,0)--(0,2)--cycle;
    \draw[step=1, gray!30] (-1,-1) grid (3,3);
    \draw[-, thick, blue] (0,0)--(2,0)--(0,2)--(0,0)node[midway, right] {};
    \draw[-, thick, blue] (0,2)--(1,0)node[midway, right] {};
\end{tikzpicture}
&
\begin{tikzpicture}
\filldraw[fill=gray!10!white,thick,draw=none] (0,0)--(3,0)--(0,3)--cycle;
    \draw[step=1, gray!30] (-1,-1) grid (4,4);
    \draw[-, thick, blue] (0,0)--(3,0)--(0,3)--(0,0)node[midway, right] {};
    \draw[-, thick, blue] (0,2)--(1,2)--(0,1)--(2,1)--(2,0)node[midway, right] {};
    \draw[-, thick, blue] (0,1)--(1,0)--(2,1)node[midway, right] {};
\end{tikzpicture}
\end{tabular}
\caption{Dual subdivisions.}
\label{fig:dualsubdivisions}
\end{figure}
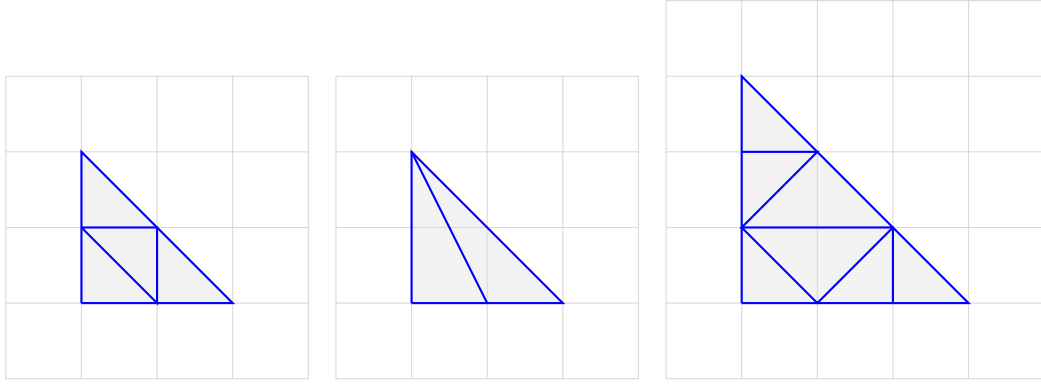
\end{ex}

\begin{exercise}
\label{exercise:balancing}
    Use the dual subdivision to show that plane tropical curves satisfy the \emph{balancing condition}: At every vertex $v\in \Gamma$ and every edge $e$ with vertex $v$, let $u_e\in \Z^2$ be the vector pointing away from $v$ in direction $e$ with entries coprime (the $u_e$ are drawn as arrows in Figure \ref{fig:balancingcondition}). Then for a fixed vertex $v$
    \[\sum_{v\in e}w(e)u_e=0.\] 
    \begin{figure}
        \begin{tabular}{cc}
            \begin{tikzpicture}
               % \filldraw[fill=gray!10!white,thick,draw=none] (0,0)--(3,0)--(0,3)--cycle;
    \draw[step=1, gray!30] (-2,-2) grid (2,2);
    \draw[-, thick, blue] (0,-2)--(0,0)--(1,2);
    \draw[-, thick, blue] (-2,0)--(2,0);
    \draw[->,thick,blue](0,0)--(1,0);
    \draw[->,thick,blue](0,0)--(0,-1);
    \draw[->,thick,blue](0,0)--(-1,0);
    \draw[->,thick,blue](0,0)--(1,2);
\node[below left, scale=0.8, blue] at (-1,0) {$2$};   
\node[below left, scale=0.8, blue] at (0,-1) {$2$};  
\node[below left, scale=0.8, black] at (0,0) {$v$};  
\end{tikzpicture}
     
             &  
          \begin{tikzpicture}
               \filldraw[fill=gray!10!white,thick,draw=none] (0,0)--(2,0)--(2,1)--(0,2)--(0,0);
    \draw[step=1, gray!30] (-1,-1) grid (3,3);
    \draw[-, thick, blue] (0,0)--(2,0)--(2,1)--(0,2)--(0,0);  
    \node[below left, scale=0.8, black] at (1,1) {$\Delta_v$}; 
\end{tikzpicture}
        \end{tabular}
        \caption{Balancing condition.}
        \label{fig:balancingcondition}
    \end{figure}
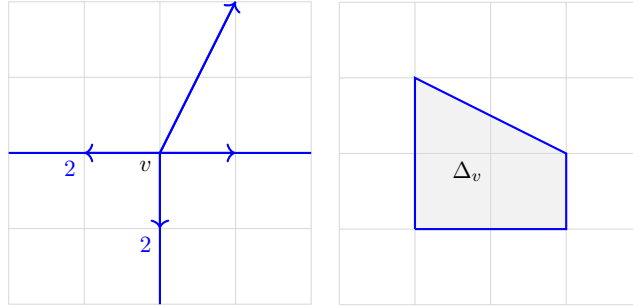
\end{exercise}

\subsection{Bézout's theorem for tropical curves}

We conclude this lecture by explaining an application of tropical geometry to enumerative geometry by showing Bézout's theorem for tropical curves.
This will imply Bézout's theorem for curves over $\Puiseux{k}$ where $k$ is algebraically closed and $\operatorname{char}k=0$.
Let's start with two algebraic curves $C_1$ and $C_2$ over $\Puiseux{k}$ defined by polyonimals $F_1,F_2\in \Puiseux{k}[z_1,z_2]$. 
Let $\Gamma_1$ and $\Gamma_2$ be the two tropical curves defined by the tropical polynomials $F_1^{\trop}$ and $F_2^{\trop}$, respectively.
Assume that $(\tp_1(t),\tp_2(t))\in C_1\cap C_2$ is an intersection point. Then the tropical curves $\Gamma_1$ and $\Gamma_2$ intersect in the tropicalization $p\coloneqq (-\val(\tp_1(t)),-\val(\tp_2(t)))\in \R^2$ of $(\tp_1(t),\tp_2(t))$. The converse also holds: if $\Gamma_1$ and $\Gamma_2$ intersect at a point $p \in \Gamma_1 \cap \Gamma_2 \subset \mathbb{R}^2$, then there exists $(\tilde{p}_1(t), \tilde{p}_2(t)) \in C_1 \cap C_2$ such that $p = (-\operatorname{val}(\tilde{p}_1(t)), -\operatorname{val}(\tilde{p}_2(t)))$. However, such a point of intersection in $C_1 \cap C_2$ need not be unique.

\begin{definition}
    $\Gamma_1$ and $\Gamma_2$ intersect \emph{tropically transversally} at $p$ if 
    \begin{itemize}
        \item $p$ is an \emph{isolated intersection point}, that is, there is a small open ball around $p$ such that $p$ is the only intersection point of $\Gamma_1$ and $\Gamma_2$, and
        \item $p$ is not a vertex of $\Gamma_1$ or $\Gamma_2$.
    \end{itemize}
\end{definition}

The union $\Gamma_1\cup \Gamma_2$ is the tropical curve defined by the tropical polynomial $(F_1\cdot F_2)^{\trop}=F_1^{\trop}\odot F_2^{\trop} $ and $p\in \Gamma_1\cup\Gamma_2$ is a $4$-valent vertex of $\Gamma_1\cup\Gamma_2$ in case $\Gamma_1$ and $\Gamma_2$ intersect tropically transversally at $p$. Thus dual to this $4$-valent vertex is a quadrilateral $\Delta_{p}$. In fact $\Delta_p$ is a parallelogram in the dual subdivision since two opposite edges in the dual quadrilateral must be parallel since both have to be orthogonal to the same edge of one of the curves, as illustrated in Figure \ref{fig:dualparallelogram}. 
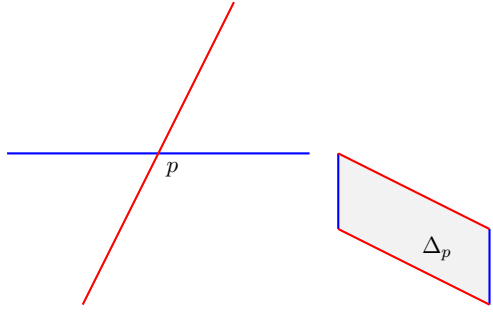
\begin{figure}
    \begin{tabular}{cc}
    \begin{tikzpicture}
    \draw[-, thick, blue] (-2,0)--(2,0)node[midway, right] {};
    \draw[-, thick, red] (-1,-2)--(1,2)node[midway, right] {};
    \node[below right, scale=0.8, black] at (0,0) {$p$};
    \end{tikzpicture} &
\begin{tikzpicture}
\filldraw[fill=gray!10!white,thick,draw=none] (0,1)--(0,2)--(2,1)--(2,0)--cycle;
    \draw[-, thick, blue] (0,1)--(0,2)node[midway, right] {};
    \draw[-, thick, blue] (2,0)--(2,1)node[midway, right] {};
    \draw[-, thick, red] (0,2)--(2,1)node[midway, right] {};
    \draw[-, thick, red] (0,1)--(2,0)node[midway, right] {};
    \node[below right, scale=0.8, black] at (1,1) {$\Delta_p$};
    \end{tikzpicture}
    \end{tabular}
        \caption{An intersection point $p$ of two tropical curves $\Gamma_1$ and $\Gamma_2$ and its dual parallelogram $\Delta_p$.}
    \label{fig:dualparallelogram}
\end{figure}

\begin{lemma}
\label{lemma:tropicalintersectionmult}
Let $p\in \Gamma_1\cap\Gamma_2$ be an intersection point and assume that $\Gamma_1$ and $\Gamma_2$ intersect tropically transversally at $p$.
    Then the number of intersection points $(\tp_1(t),\tp_2(t))\in C_1\cap C_2$ counted with multiplicity such that $(-\val(\tp_1(t)),-\val(\tp_2(t)))=p\in \R^2$ equals $\operatorname{Area}(\Delta_{p})$.
\end{lemma}

This motivates the following definition.
\begin{definition}
    Let $\Gamma_1$ and $\Gamma_2$ be two tropical curves and let $p\in \Gamma_1\cap \Gamma_2$. Then the \emph{intersection multiplicity} of $\Gamma_1$ and $\Gamma_2$ at $p$ is given by
    \[\mult_p(\Gamma_1,\Gamma_2)\coloneqq \operatorname{Area}(\Delta_{p})\]
    where $\Delta_{p}$ is the parallelogram dual to the $4$-valent vertex $p$ of the tropical curve $\Gamma_1\cup\Gamma_2$ in the dual subdivision $\operatorname{DS}(\Gamma_1\cup\Gamma_2)$.
\end{definition}
%Now the proof of Bézout's theorem for curves over $\Puiseux{k}$ follows from the following Bézout theorem for tropical curves and Lemma \ref{lemma:tropicalintersectionmult}.

\begin{ex}
\label{ex:intersectionoftropicallinewithconic}
Figure \ref{fig:conicandline} shows a tropical line (in blue) intersecting a tropical conic (in red) in two different ways. On the left there are two intersections and both have intersection multiplicity $1$ and on the right there is only one intersection point with intersection multiplicity $2$.

\begin{figure}
    \begin{tabular}{cccc}
    \begin{tikzpicture}[scale=0.8]
    \draw[-, thick, red] (2,3)--(0,1)--(0,0)--(1,-1)--(2,-1)--(4,1);
    \draw[-, thick, red] (0,1)--(-1,1)node[midway, right] {};
    \draw[-, thick, red] (0,0)--(-1,0)node[midway, right] {};
    \draw[-, thick, red] (1,-1)--(1,-2)node[midway, right] {};
    \draw[-, thick, red] (2,-1)--(2,-2)node[midway, right] {};
        \draw[-, thick, blue] (-1,0.5)--(1.5,0.5)--(1.5,-2)node[midway, right] {};
    \draw[-, thick, blue] (1.5,0.5)--(3,2)node[midway, right] {};

    \node[below right, scale=0.8, black] at (0,0.5) {$p$};
    \node[below right, scale=0.8, black] at (1.5,-1) {$q$};
    \end{tikzpicture} &
\begin{tikzpicture}[scale=0.8]
\filldraw[fill=gray!10!white,thick,draw=none] (0,0)--(0,3)--(3,0)--cycle;
    \draw[-, thick, red] (0,3)--(1,2)--(0,2)--(0,3)node[midway, right] {};
    \draw[-, thick, blue] (0,2)--(0,1)node[midway, right] {};
    \draw[-, thick, blue] (1,2)--(1,1)node[midway, right] {};
    \draw[-, thick, red] (0,1)--(1,1)--(1,0)--(0,0)--(0,1)node[midway, right] {};
    \draw[-, thick, red] (2,0)--(3,0)--(2,1)--(2,0)node[midway, right] {};
    \draw[-, thick, blue] (2,1)--(1,2)node[midway, right] {};
    \draw[-, thick, red] (0,0)--(1,1)node[midway, right] {};
    \draw[-, thick, blue] (2,0)--(1,0)node[midway, right] {};
    \draw[-, thick, blue] (2,1)--(1,1)node[midway, right] {};
    \node[scale=0.8, black] at (0.5,1.5) {$\Delta_p$};
    \node[scale=0.8, black] at (1.5,0.5) {$\Delta_q$};
    \end{tikzpicture}& 
    \begin{tikzpicture}[scale=0.8]
    \draw[-, thick, red] (2,3)--(0,1)--(0,0)--(1,-1)--(2,-1)--(4,1);
    \draw[-, thick, red] (0,1)--(-1,1)node[midway, right] {};
    \draw[-, thick, red] (0,0)--(-1,0)node[midway, right] {};
    \draw[-, thick, red] (1,-1)--(1,-2)node[midway, right] {};
    \draw[-, thick, red] (2,-1)--(2,-2)node[midway, right] {};
        \draw[-, thick, blue] (-1,-1)--(0,-1)--(0,-2)node[midway, right] {};
    \draw[-, thick, blue] (0,-1)--(3,2)node[midway, right] {};

    \node[below right, scale=0.8, black] at (0,-1) {$p$};
    
    \end{tikzpicture}&
\begin{tikzpicture}[scale=0.8]
\filldraw[fill=gray!10!white,draw=none] (0,0)--(0,3)--(3,0)--cycle;
    \draw[-, thick, red] (0,2)--(0,3)--(1,2)--(0,2)--(0,1)--(1,2)node[midway, right] {};
    \draw[-, thick, red] (2,0)--(1,0)--(2,1)node[midway, right] {};
    \draw[-, thick, blue] (0,1)--(1,0)--(0,0)--(0,1)node[midway, right] {};
    \draw[-, thick, red] (2,0)--(3,0)--(2,1)--(2,0)node[midway, right] {};
    \draw[-, thick, blue] (2,1)--(1,2)node[midway, right] {};

    \node[scale=0.8, black] at (1,1) {$\Delta_p$};

    \end{tikzpicture}
    \end{tabular}
        \caption{A tropical conic and a tropical line intersecting in two points with multiplicity $1$ on the left and in one point with multiplicity $2$ on the right.}
    \label{fig:conicandline}
\end{figure}
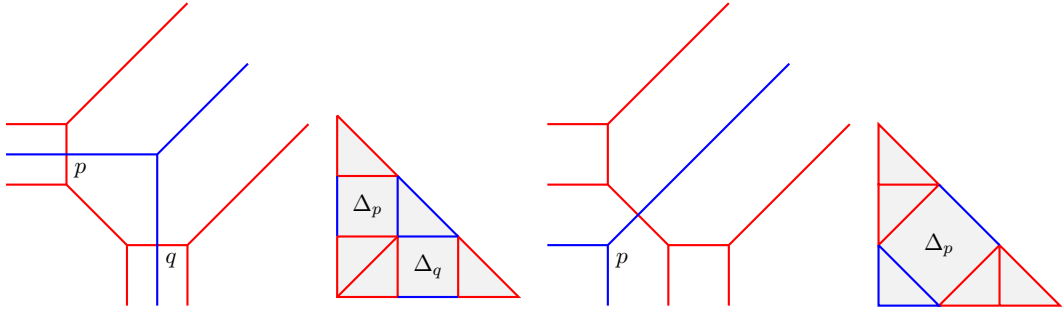

\end{ex}
Let 
\[\Delta_d\coloneqq \operatorname{Conv}\{(0,0),(d,0),(0,d)\}.\]
This is the Newton polygon of a (tropical) polyomial $\bigoplus_{0\le i+j\le d} a_{ij}\odot x^{\odot i}\odot y^{\odot j}$ of degree $d$ with coefficients $a_{ij}\neq-\infty$.
\begin{thm}[Bézout theorem for tropical curves]
\label{thm:Bezoutfortropcurves}
    Let $\Gamma_1$ and $\Gamma_2$ be two tropical curves with Newton polygons $\Delta_{d_1}$ respectively $\Delta_{d_2}$. 
    Then 
    \[\sum_{p\in \Gamma_1\cap \Gamma_2}\mult_{p}(\Gamma_1,\Gamma_2)=d_1\cdot d_2.\]
\end{thm}
\begin{proof}
The tropical curve $\Gamma_1\cup \Gamma_2$ has Newton polygon $\Delta_{d_1+d_2}$. The dual subdivision of $\Gamma_1\cup \Gamma_2$ consists of the dual subdivision of $\Gamma_1$, $\Gamma_2$ and the parallelograms corresponding to the intersection points as indicated by the different colors in Figure \ref{fig:twotropcurvesintersecting}. Thus
    \begin{align*}
        \sum_{p\in \Gamma_1\cap \Gamma_2}\mult_{p}(\Gamma_1,\Gamma_2)&=\Area(\Delta_{d_1+d_2})-\Area(\Delta_{d_1})-\Area(\Delta_{d_2})\\
        &=\frac{(d_1+d_2)^2}{2}-\frac{d_1^2}{2}-\frac{d_2^2}{2}=d_1\cdot d_2
    \end{align*}

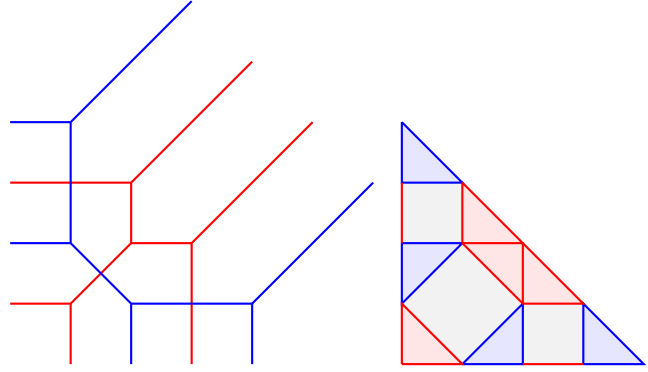
\begin{figure}
    \begin{tabular}{cc}
    \begin{tikzpicture}[scale=0.8]
    \draw[-, thick, red] (-1,0)--(0,0)--(1,1)--(1,2)--(3,4);
    \draw[-, thick, red] (0,0)--(0,-1);
    \draw[-, thick, red] (1,1)--(2,1)--(2,-1);
    \draw[-, thick, red] (2,1)--(4,3);
    \draw[-, thick, red] (1,2)--(-1,2);
        \draw[-, thick, blue] (2,5)--(0,3)--(0,1)--(1,0)--(3,0)--(5,2);
        \draw[-, thick, blue] (0,3)--(-1,3);
        \draw[-, thick, blue] (0,1)--(-1,1);
        \draw[-, thick, blue] (1,0)--(1,-1);
        \draw[-, thick, blue] (3,0)--(3,-1);
    
    %\node[below right, scale=0.8, black] at (0,0.5) {$p$};
    %\node[below right, scale=0.8, black] at (1.5,-1) {$q$};
    \end{tikzpicture} &
\begin{tikzpicture}[scale=0.8]
\filldraw[fill=gray!10!white,draw=none] (0,2)--(0,3)--(1,3)--(1,2)--(0,2)--cycle;
\filldraw[fill=gray!10!white,draw=none] (2,0)--(3,0)--(3,1)--(2,1)--(2,0)--cycle;
\filldraw[fill=gray!10!white,draw=none] (0,1)--(1,0)--(2,1)--(1,2)--(0,1)--cycle;
    \filldraw[fill=blue!10!white, thick,draw=blue] (0,4)--(1,3)--(0,3)--(0,4);
    \filldraw[fill=blue!10!white, thick,draw=blue] (0,2)--(0,1)--(1,2)--(0,2);
    \filldraw[fill=blue!10!white, thick,draw=blue] (1,0)--(2,0)--(2,1)--(1,0);
    \filldraw[fill=blue!10!white, thick,draw=blue] (3,0)--(4,0)--(3,1)--(3,0);
    \filldraw[fill=red!10!white, thick,draw=red] (1,3)--(3,1)--(2,1)--(1,2)--(1,3);
    \filldraw[fill=red!10!white, thick,draw=red] (0,0)--(1,0)--(0,1)--(0,0);
    \draw[-, thick,draw=red] (1,2)--(2,2)--(2,1);
    \draw[-, thick,draw=red] (2,0)--(3,0);
    \draw[-, thick,draw=red] (0,2)--(0,3);
    %\filldraw[fill=red!10!white, thick,draw=red] (1,0)--(2,0)--(2,1)--(1,0);
    %\filldraw[fill=red!10!white, thick,draw=red] (3,0)--(4,0)--(3,1)--(3,0);
    %\node[scale=0.8, black] at (0.5,1.5) {$\Delta_p$};
    %\node[scale=0.8, black] at (1.5,0.5) {$\Delta_q$};
    \end{tikzpicture}
    \end{tabular}
        \caption{Two tropical degree $2$ curves intersecting with the dual subdivision of their union.}
    \label{fig:twotropcurvesintersecting}
\end{figure}
\end{proof}
Now together with Lemma \ref{lemma:tropicalintersectionmult} this implies
\begin{cor}[Bézout over $\Puiseux{k}$]
    Let $C_1$ and $C_2$ be two algebraic curves in $\P^2_{\Puiseux{k}}$ of degree $d_1$ respectively $d_2$. Then
    \[\sum_{\tp\in C_1\cap C_2}\mult_{\tp}(C_1,C_2)=d_1\cdot d_2.\]
\end{cor}
\begin{rmk}
    This also works for tropical curves with Newton polygons other than $\Delta_d$. In this case 
    $\sum_{p\in \Gamma_1\cap \Gamma_2}\mult_{p}(\Gamma_1,\Gamma_2)$ equals the \emph{mixed volume} of the two Newton polygons. This again implies the analogous algebraic statement over $\Puiseux{k}$ known as the Bernstein-Kushnirenko theorem. All of this also works in higher dimensions and dimension $1$, that is not only for (tropical) curves but also for (tropical) hypersurfaces.
\end{rmk}

\section{Tropical enumerative geometry}

In the last lecture, we prove Bézout's theorem for tropical curves with 
multiplicities in $\GW(k)$
\cite{JaramilloPuentesPauliBezout}. 
To achieve this, we introduce the notion of \emph{enriched tropical curves} over $k$. 
During the tropicalization process, too much information is lost to define an enriched intersection multiplicities valued in $\GW(k)$; enriched tropical curves retain 
precisely the additional data needed to recover this information. 
This leads to a purely combinatorial proof of Bézout's theorem for tropical curves 
with multiplicities in $\GW(k)$.

We finish this lecture series with a very powerful application of tropical geometry in enumerative geometry. 
The major breakthrough in tropical enumerative geometry is Mikhalkin's 
\emph{tropical correspondence theorem} \cite{Mikhalkin}, 
which states that the count of plane degree~$d$, genus~$g$ curves passing through 
a given configuration of points equals the weighted count of tropical plane curves of the same 
degree and genus with point conditions, where the tropical curves are counted 
with appropriate multiplicities. 
There are also tropical correspondence theorems for counting real plane rational degree~$d$ curves with 
point conditions \cite{Mikhalkin,Shustin}. 
Very recent work 
\cite{JaramilloPuentesPauliCorrespondence,JPMPRnew} extends this to the 
version over an (almost) arbitrary field $k$, establishing a tropical correspondence theorem in this 
refined setting.

\subsection{Bézout for tropical curves in $\GW(k)$}
In the last lecture we will bring together the two worlds of tropical geometry and enumerative geometry in $\GW(k)$.
To avoid non-separable field extensions and zeros of our polynomials not contained in $\Puiseux{\overline{k}}$, we assume that the characteristic of $k$ is either $0$ or big, that is larger than the degrees of the polynomials involved.
Let $\Puiseux{k}$ be the field of Puiseux series over $k$.

The exercise below demonstrates the canonical isomorphism $\GW(k) \cong \GW(\Puiseux{k})$, offering a natural justification for weighted counts with weights in $\GW(k)$ are compatible with the framework of tropical geometry.

 \begin{exercise}
 \label{exercise:GWPuiseux}
\begin{enumerate}
    \item There is a bijection
     \[\nicefrac{\Puiseux{k}^\times}{(\Puiseux{k}^\times)^2}\cong \nicefrac{k^\times}{(k^\times)^2}, \; \ta(t)=a_0t^{q_0}+h.o.t.\mapsto a_0\]
     \item Conclude that $\GW(\Puiseux{k})\rightarrow \GW(k)$ defined by $\langle a_0t^{q_0}+h.o.t.\rangle\mapsto \langle a_0\rangle$ is an isomorphism by checking that this map respects the relations in the Grothendieck-Witt rings.
\end{enumerate}
 \end{exercise}

We have the following Bézout theorem for curves over $\Puiseux{k}$ (see Example \ref{ex:bezoutoveranyfield}) which we will reprove using tropical geometry.

\begin{thm}
\label{thm:bezoutforcurvesoverpuiseux}
    Let $C_1$ and $C_2$ be curves in $\P^2_{\Puiseux{k}}$ of degree $d_1$ respectively $d_2$ with $d_1+d_2$ odd which meet transversally and with all intersection points contained in some $\A^2_{\Puiseux{k}}$ where $k$ is a perfect field of characteristic larger than $d_1+d_2$. Then as subsets of $\A^2_{\Puiseux{k}}$ one can write $C_1=V(F_1)$ and $C_2=V(F_2)$ with $F_1,F_2\in \Puiseux{k}[z_1,z_2]$. It then holds that 
    \[\sum_{\tp\in C_1\cap C_2}\Tr_{\kappa(\tp)/\Puiseux{k}}\langle \det \operatorname{Jac} (F_1,F_2)(\tp)\rangle=\frac{d_1\cdot d_2}{2}h\in \GW(\Puiseux{k})\overset{\ref{exercise:GWPuiseux}}{\cong} \GW(k)\]
    where $\kappa(\tp)$ is the residue field of $\tp$.
\end{thm}
 Let $\Gamma_1$ and $\Gamma_2$ be the two tropical curves defined by $F_1^{\trop}$ and $F_2^{\trop}$.
To prove Theorem \ref{thm:bezoutforcurvesoverpuiseux} we need an analog of Lemma \ref{lemma:tropicalintersectionmult} yielding a tropical intersection multiplicity valued in $\GW(k)\cong \GW(\Puiseux{k})$.

However, by passing from $C_i$ to $\Gamma_i$ for $i=1,2$ we lose too much information when $k$ is not algebraically closed.
 We need to ``enrich" our tropical curves. Exercise \ref{exercise:GWPuiseux} indicates that we only need to remember the initial coefficients of the coefficients of $F_1$ and $F_2$ up to squares in $k^\times$.

%From this exercise one can deduce the following Lemma by checking that the relations in the Grothendieck-Witt rings are respected.
%\begin{lemma}
%$\GW(\Puiseux{k})\rightarrow \GW(k)$ defined by $\langle a(t)\rangle\mapsto a_0$ where $a(t)=a_0t^{n_0}+h.o.t.$ defines an isomorphism of Grothendieck-Witt rings.
%\end{lemma}

Let $F(z_1,z_2)=\sum \ta_{ij}(t)z_1^i z_2^j\in \Puiseux{k}[z_1,z_2]$.
Recall that the connected components of the complement of a tropical curve correspond to the tropical monomials in the tropical polynomial $F^{\trop}$ where the maximum is attained, and these monomials correspond to monomials of $F$. We want to remember the coefficients $\ta_{ij}(t)=a_{ij}t^{\val(\ta_{ij})}+h.o.t.$ in $F$ up to squares in $\Puiseux{k}^\times$ which by the above exercise is equal to the class $[a_{ij}]$ of $a_{ij}$ in $k^\times/(k^\times)^2$ and we do this by decorating the components of the complement of the associated tropical curve with $a_{ij}$. Recall that the components of the complement of the tropical curve correspond to the vertices in the dual subdivision, so we also label these as illustrated in Figure \ref{figure:enriched tropical curve} which leads to the following definition. 
\begin{figure}
\begin{tabular}{cc}
%\begin{tikzpicture}[scale=0.8]
%\draw[blue,thick] (-3,0)--(0,0)--(0,-3);
%\draw[blue,thick] (0,0)--(2,2);
%\node[blue] at (2,-0.5) {$a_{10}$};
%\node[blue] at (-0.5,2) {$a_{01}$};
%\node[blue] at (-1,-1) {$a_{00}$};

%\end{tikzpicture}&
\begin{tikzpicture}[scale=0.6]
\draw[blue,thick] (-1,2)--(2,2)--(2,-1);
\draw[blue,thick] (2,2)--(3,3)--(4,3)--(4,-1);
\draw[blue,thick] (4,3)--(6.5,5.5);
\draw[blue,thick] (3,3)--(3,4)--(5.5,6.5);
\draw[blue,thick] (3,4)--(-1,4);
\node[blue] at (3,0) {$[a_{10}]$};
\node[blue] at (0,3) {$[a_{01}]$};
\node[blue] at (0,0) {$[a_{00}]$};
\node[blue] at (5,1) {$[a_{20}]$};
\node[blue] at (1,5) {$[a_{02}]$};
\node[blue] at (5,5) {$[a_{11}]$};
\end{tikzpicture}&
\begin{tikzpicture}[scale=0.75]
\filldraw[fill=gray!10!white,thick,draw=none] (0,0)--(4,0)--(0,4)--cycle;
\draw[blue,thick] (0,0)--(4,0)--(0,4)--cycle;
\draw[blue,thick] (2,0)--(0,2)--(2,2)--cycle;
\node[blue] at (2,-0.4) {$[a_{10}]$};
\node[blue] at (-0.45,2) {$[a_{01}]$};
\node[blue] at (-0.3,-0.3) {$[a_{00}]$};
\node[blue] at (4.35,-0.3) {$[a_{20}]$};
\node[blue] at (-0.3,4.35) {$[a_{02}]$};
\node[blue] at (2.35,2.35) {$[a_{11}]$};
\end{tikzpicture} 
\end{tabular}
\caption{An enriched tropical curve and its enriched dual subdivision.}
\label{figure:enriched tropical curve}
\end{figure}
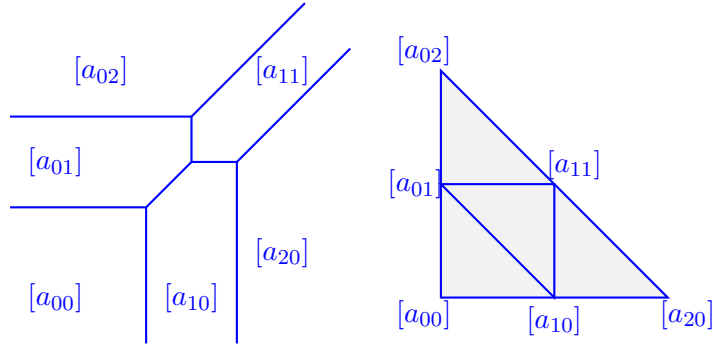

\begin{definition}
    An \emph{enriched tropical curve} over $k$ consists of a tropical curve $\Gamma$ together with $[a_{ij}]\in k^\times/(k^\times)^2$ for each component of $\R^2\setminus \Gamma$ or equivalently vertex $(i,j)$ in the dual subdivision of $\Gamma$. We call the $[a_{ij}]$ \emph{coefficients} at the component/vertex.
\end{definition}

%{\color{red} add Nathan's phase structures as a generalization of this}

      Given two enriched tropical curves we also want to equip the union of the two underlying tropical curves with an enrichment. We do this in the following way.
The coefficient of a component of the union is given by the product of the coefficients of the respective components of the two tropical curves, as shown for the example of the union of two enriched tropical curves in Figure \ref{fig:unionofenrichedcurves}.
This is motivated by the following observation. 
    Let $F_1, F_2 \in \Puiseux{k}[z_1, z_2]$.  
    Both $F_1$ and $F_2$ give rise to an enriched tropical curve $\Gamma_1$ respectively $\Gamma_2$ with coefficients $[a_{ij}]$ respectively $[b_{ij}]$, obtained from the initials of the coefficients of $F_1$ respectively $F_2$. 
    The product $F_1 \cdot F_2$ also determines an enriched tropical curve whose underlying (non-enriched) tropical curve is the union $\Gamma_1 \cup \Gamma_2$, and whose coefficients are exactly those obtained from $[a_{ij}]$ and $[b_{ij}]$ in this manner (see \cite[Lemma 3.8]{JaramilloPuentesPauliBezout}).

\begin{figure}
    \begin{tabular}{ccc}
    \begin{tikzpicture}[scale=2]
\draw[blue,thick] (-1,0)--(0,0)--(0,-1);
\draw[blue,thick] (0,0)--(1,1);
\node[blue] at (1,0) {$[a_{10}]$};
\node[blue] at (0,1) {$[a_{01}]$};
\node[blue] at (-0.5,-0.5) {$[a_{00}]$};

    \end{tikzpicture}&
        \begin{tikzpicture}[scale=2]
\draw[red,thick] (-1,0)--(0,0)--(0,-1);
\draw[red,thick] (0,0)--(1,1);
\node[red] at (1,0) {$[b_{10}]$};
\node[red] at (0,1) {$[b_{01}]$};
\node[red] at (-0.5,-0.5) {$[b_{00}]$};

    \end{tikzpicture}&
     \begin{tikzpicture}[scale=2.6]
     \draw[blue,thick] (-1.5,0.5)--(-0.5,0.5)--(-0.5,-1);
\draw[blue,thick] (-0.5,0.5)--(0.5,1.5);
\draw[red,thick] (-1.5,0)--(0,0)--(0,-1);
\draw[red,thick] (0,0)--(1,1);
\node[purple] at (0.5,0) {$[a_{10}b_{10}]$};
\node[purple] at (-1,1) {$[a_{01}b_{01}]$};
\node[purple] at (-1,-0.5) {$[a_{00}b_{00}]$};
\node[purple] at (-0.1,0.5) {$[a_{10}b_{01}]$};
\node[purple] at (-1,0.25) {$[a_{00}b_{01}]$};
\node[purple] at (-0.24,-0.5) {$[a_{10}b_{00}]$};

    \end{tikzpicture}
    \end{tabular}
    \caption{The union of two enriched tropical curves.}
    \label{fig:unionofenrichedcurves}
\end{figure}
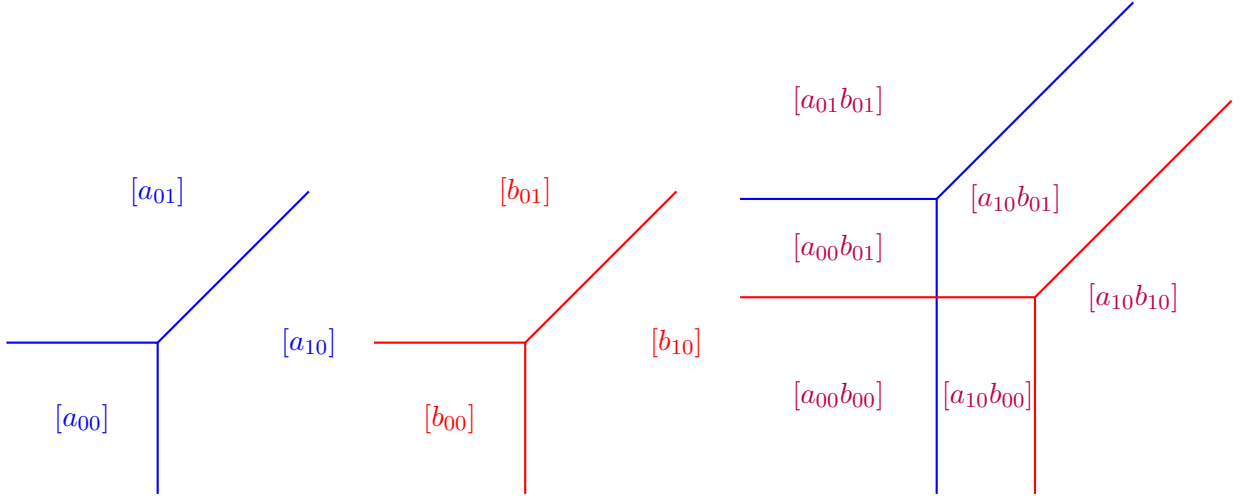

\begin{definition}[Tropical $\A^1$-intersection multiplicity]
\label{def:enrichedtropintmult}
Let $F_1,F_2\in \Puiseux{k}[z_1,z_2]$ 
and let $\Gamma_1$ and $\Gamma_2$ be the two associated enriched tropical curves. Assume $\Gamma_1$ and $\Gamma_2$ intersect tropically transversally at $p$.
We define the \emph{tropical $\A^1$-intersection multiplicity} of $\Gamma_1$ and $\Gamma_2$ at $p$ to be
    \[\mult^{\A^1}_p(\Gamma_1,\Gamma_2)\coloneqq  \sum_{\tp}\operatorname{Tr}_{\kappa(\tp)/\Puiseux{k}} (\langle\det \operatorname{Jac}(F_1,F_2)(\tp)\rangle)\in \GW(\Puiseux{k})\]
        where the sum runs over the common zeros $\tp$ of $F_1$ and $F_2$, which tropicalize to $p$. Here, $\kappa(\tp)$ denotes the residue field of $\tp$.
\end{definition} 

We say that a lattice point $v\in \Z^2$ is \emph{odd} if both entries are odd.
The following theorem is the generalization of Lemma \ref{lemma:tropicalintersectionmult}.
\begin{thm}[Jaramillo Puentes-Pauli \cite{JaramilloPuentesPauliBezout}]
\label{thm:enrichedintersectionmult}
Let $\Gamma_1$ and $\Gamma_2$ be enriched tropical curves of degree $d_1$ respectively $d_2$.
Assume that $k$ is a field of characteristic $0$ or greater than the sum of the degrees of $\Gamma_1$ and $\Gamma_2$.
     Let $\Delta_p$ be the parallelogram in $\operatorname{DS}(\Gamma_1\cup \Gamma_2)$ dual to $p\in \Gamma_1\cap \Gamma_2$. Then
    \[\operatorname{mult}^{\A^1}_p(\Gamma_1,\Gamma_2)=\frac{\operatorname{Area}(\Delta_p)-\#\{\text{$v$ odd vertex of $\Delta_p$}\}}{2}\cdot h+\sum_{\text{$v$ odd vertex of $\Delta_p$}}\langle \epsilon_p(v)a_v\rangle\]
    where
    $[a_v]\in k^\times/(k^\times)^2$ is the coefficient of the vertex $v$ and
    \[\epsilon_p(v)=\begin{cases}+1 &\text{if one starts at an edge dual to an edge of $\Gamma_1$} \\
    -1 &\text{if one starts at an edge dual to an edge of $\Gamma_2$} 
    \end{cases}\]
    when walking around $v$ inside of $\Delta_p$ anticlockwise as illustrated in Figure \ref{fig:epsilons}.
\end{thm}
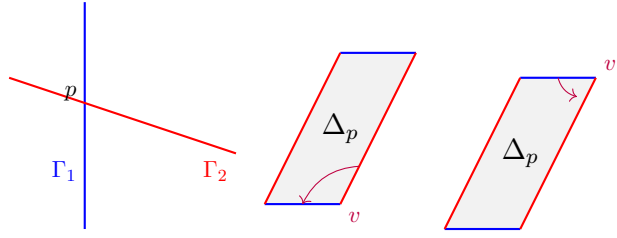
\begin{figure}
    \begin{tabular}{ccc}
    \begin{tikzpicture}
        \draw[thick,blue] (0,-1)--(0,2);
        \draw[thick,red] (-1,1)--(2,0);
        \node[below left,scale=0.8,black] at (0,1) {$p$};
        \node[below left,scale=0.8,blue] at (0,0) {$\Gamma_1$};
        \node[below left,scale=0.8,red] at (2,0) {$\Gamma_2$};
    \end{tikzpicture}
    &
    \begin{tikzpicture}
        \filldraw[fill=gray!10!white,thick,draw=none] (0,0)--(1,0)--(2,2)--(1,2)--cycle;
        \draw[blue,thick](0,0)--(1,0);
        \draw[blue,thick](1,2)--(2,2);
        \draw[red,thick](1,0)--(2,2);
        \draw[red,thick](0,0)--(1,2);
        \node[black] at (1,1) {$\Delta_p$};
         \node[below right, scale=0.8, purple] at (1,0) {$v$};
         \draw[->, bend right=30,purple] (1.25,0.5) to (0.5,0);
    \end{tikzpicture}&
    \begin{tikzpicture}
        \filldraw[fill=gray!10!white,thick,draw=none] (0,0)--(1,0)--(2,2)--(1,2)--cycle;
        \draw[blue,thick](0,0)--(1,0);
        \draw[blue,thick](1,2)--(2,2);
        \draw[red,thick](1,0)--(2,2);
        \draw[red,thick](0,0)--(1,2);
        \node[black] at (1,1) {$\Delta_p$};
        \node[above right, scale=0.8, purple] at (2,2) {$v$};
        \draw[->, bend right=30,purple] (1.5,2) to (1.75,1.75);
    \end{tikzpicture}
    \end{tabular}
    \caption{For $v$ in the middle $\epsilon_p(v)=-1$ and for $v$ on the right $\epsilon_p(v)=+1$.}
    \label{fig:epsilons}
\end{figure}
\begin{ex}
\label{ex:enrichedintersectionmult}
    Let's go back to Example \ref{ex:intersectionoftropicallinewithconic}.
    Let $\Gamma_1$ be the tropical line (in blue) and $\Gamma_2$ be the tropical conic (in red) in Figure \ref{fig:conicandline}. There is only one odd lattice point inside of $\Delta_3$, namely $(1,1)$. Let $[a_{11}]\in \nicefrac{k^\times}{(k^\times)^2}$ be the coefficient of $(1,1)$.
    For the intersection on the left, the tropical $\A^1$-intersection multiplicities are as follows
    \[\mult_p(\Gamma_1,\Gamma_2)=\langle a_{11}\rangle, \; \mult_q(\Gamma_1,\Gamma_2)=\langle-a_{11}\rangle.\]
    To compute the tropical $\A^1$-intersection multiplicity on the right note that $\Delta_p$ has no odd vertices (there is one odd lattice point in the interior but this does not count). The area of $\Delta_p$ on the right is $2$, so using the formula in Theorem \ref{thm:enrichedintersectionmult} we get
    \[\mult_p^{\A^1}(\Gamma_1,\Gamma_2)=\frac{2}{2}h=h.\]
\end{ex}

The proof of Theorem \ref{thm:enrichedintersectionmult} is straightforward, yet tedious: One computes $\mult_p^{\A^1}(\Gamma_1,\Gamma_2)$ as defined in Definition \ref{def:enrichedtropintmult} and identifies it with the combinatorial formula in Theorem \ref{thm:enrichedintersectionmult}.

\begin{ex}
    Let's verify Theorem \ref{thm:enrichedintersectionmult} for the case of two lines intersecting.
    Let $F_1=\ta_1(t)z_1+\tb_1(t)z_2+\tc_1(t)\in \Puiseux{k}[z_1,z_2]$ and $F_2=\ta_2(t)z_1+\tb_2(t)z_2+\tc_2(t)\in \Puiseux{k}[z_1,z_2]$.
    Then the vertex of the tropical line $\Gamma_1$ associated with $F_1$ lies on the line 
    \[y=x+\val(\tb_1)-\val(\ta_1)\]
    and the vertex of the tropical line $\Gamma_2$ associated with $F_2$ lies on the line 
    \[y=x+\val(\tb_2)-\val(\ta_2).\]

    Furthermore,
    the coefficient of $z_1\cdot z_2$ in $F_1\cdot F_2$ is $\ta_1(t)\tb_2(t)+\ta_2(t)\tb_1(t)$.
    Let $a_1$, $a_2$, $b_1$ and $b_2$ be the lowest non-vanishing coefficients of $\ta_1(t)$, $\ta_2(t)$, $\tb_1(t)$ and $\tb_2(t)$, respectively. 
    Then the coefficient of the lowest power of $t$ of $\ta_1(t)\tb_2(t)+\ta_2(t)\tb_1(t)$ is 
    \begin{itemize}
        \item $a_1b_2$ if $\val(\ta_1)+\val(\tb_2)<\val(\ta_2)+\val(\tb_1)$,
        \item $a_2b_1$ if $\val(\ta_1)+\val(\tb_2)>\val(\ta_2)+\val(\tb_1)$.
    \end{itemize}
    In case $\val(\ta_1)+\val(\tb_2)=\val(\ta_2)+\val(\tb_1)$, the two tropical lines share an edge and thus do not intersect tropically transversally. So we do not need to consider this case.
    Note that this agrees with the coefficient of the odd vertex $(1,1)$ of the dual subdivision of $\Gamma_1\cup \Gamma_2$ in both cases. 
    
    Let's compute the enriched intersection multiplicity.
    \[\langle\det \operatorname{Jac} (F_1,F_2)\rangle=\langle \ta_1(t)\tb_2(t)-\ta_2(t)\tb_1(t)\rangle\in \GW(\Puiseux{k})\]
    Under the isomorphism from Exercise \ref{exercise:GWPuiseux} this agrees with 
    \begin{itemize}
        \item $\langle a_1b_2\rangle$ if $\val(\ta_1)+\val(\tb_2)<\val(\ta_2)+\val(\tb_1)$,
        \item $\langle-a_2b_1\rangle$ if $\val(\ta_1)+\val(\tb_2)>\val(\ta_2)+\val(\tb_1)$.
    \end{itemize}
    Now it is easy to check that this is indeed equal to the combinatorial formula in Theorem \ref{thm:enrichedintersectionmult}, the minus in the second case is the sign $\epsilon_p(v)$.
\end{ex}

\begin{cor}[Tropical Bézout for enriched tropical curves]
\label{cor:tropicalenrichedbezout}
Let $\Gamma_1$ and $\Gamma_2$ be enriched tropical curves of degree $d_1$ respectively $d_2$.
Further assume that $k$ is a field of characteristic $0$ or greater than the sum of the degrees of $\Gamma_1$ and $\Gamma_2$.
Assume $k$ is a field of characteristic not equal to $2$.
    Further, assume $d_1+d_2\equiv 1\mod 2$. Then
        \[\sum_{p\in \Gamma_1\cap \Gamma_2} \operatorname{mult}^{\A^1}_p(\Gamma_1,\Gamma_2)=\frac{d_1\cdot d_2}{2}\cdot h\in \GW(k).\]
\end{cor}
\begin{proof}

            First, observe that 
            $d_1+d_2$ odd implies that there are no odd lattice points on $\partial \Delta_{d_1+d_2}$. In other words, all odd lattice points lie in the interior of $\Delta_{d_1+d_2}$.
            Let $v$ be a lattice point in the interior of $\Delta_{d_1+d_2}$ (for example an odd lattice point). Then, if you make a full turn around this point, observe that you change color exactly when $v$ is a vertex of a parallolgram corresponding to an intersection point and also observe that the order of color change determines the sign $\epsilon_p(v)$. Therefore:
            \begin{itemize}
                \item The number of parallelograms in $\operatorname{DS}(\Gamma_1\cup\Gamma_2)$ corresponding to an intersection of $\Gamma_1$ and $\Gamma_2$ with vertex $v$ is even.
                \item $\#\{p\in \Gamma_1\cap \Gamma_2:$ $v$ vertex of $\Delta_p$, $\epsilon_p(v)=+1\}
                =\#\{p\in \Gamma_1\cap \Gamma_2:$ $v$ vertex of $\Delta_p$, $\epsilon_p(v)=-1\}$.
            \end{itemize}
            So in $\operatorname{W}(k)=\frac{\GW(k)}{\Z\cdot h}$ we have
            \[\sum_{p\in \Gamma_1\cap \Gamma_2}\mult_p^{\A^1}(\Gamma_1,\Gamma_2)=\sum_{\text{$v$ odd}}\frac{\#\{p\in \Gamma_1\cap \Gamma_2:v\in \Delta_p\text{ a vertex}\}}{2}(\langle a_v\rangle+\langle-a_v\rangle).\]
            
            Now recall from Exercise \ref{exercise:hyperbolicform} that $\langle a_v\rangle+\langle-a_v\rangle=h$ and thus in $\operatorname{W}(k)$ we get
            \[\sum_{p\in \Gamma_1\cap \Gamma_2}\mult_p^{\A^1}(\Gamma_1,\Gamma_2)=0.\]

             Recall from Remark \ref{rmk:wittandrankdeterminegw} that when $\operatorname{char}k\neq 2$ an element of $\GW(k)$ is uniquely determined by its image in $\operatorname{W}(k)$ and its rank. Thus in $\GW(k)$
             \[\sum_{p\in \Gamma_1\cap \Gamma_2}\mult_p^{\A^1}(\Gamma_1,\Gamma_2)=\frac{\rk\left(\sum_{p\in \Gamma_1\cap \Gamma_2}\mult_p^{\A^1}(\Gamma_1,\Gamma_2)\right)}{2}h\overset{\ref{thm:Bezoutfortropcurves}}{=}\frac{d_1\cdot d_2}{2}h.\]
    \end{proof}
    \begin{rmk}
    Note that the proof of Corollary \ref{cor:tropicalenrichedbezout} fails when $d_1+d_2$ is even which is exactly when the vector bundle $\glob(d_1)\oplus \glob(d_2)\rightarrow \P^2_k$ is not relatively orientable. So this non-orientability is reflected in the combinatorics of the corresponding tropical curves.
    
    %This seems to be no coincidence: For example one could also intersect two curves in $\P^1_k\times \P^1_k$ of bidegrees $(d_1,e_1)$ and $(d_2,e_2)$. Such curves tropicalize to tropical curves with Newton polygons $\operatorname{Conv}\{(0,0),(d_i,0),(0,e_i),(d_i,e_i)\}$ for $i=1,2$ and the Newton polygon of the union of these tropical curves is $\operatorname{Conv}\{(0,0),(d_1+d_2,0),(0,e_1+e_2),(d_1+d_2,e_1+e_2)\}$ which has no odd lattice points on the boundary (this is what is needed in the proof of Corollary \ref{cor:tropicalenrichedbezout}) if and only if the vector bundle $\glob(d_1,e_1)\oplus \glob(d_2,e_2)\rightarrow \P^1_k\times \P^1_k$ is relatively orientable.
    
    %However, it is an open question whether this observation holds in general \cite[Conjecture 6.16]{JaramilloPuentesPauliBezout}.
\end{rmk}

\begin{ex}
    Let's continue Example \ref{ex:intersectionoftropicallinewithconic} and Example \ref{ex:enrichedintersectionmult} and check whether we get the correct result. As computed in Example \ref{ex:enrichedintersectionmult} for the left picture in Figure \ref{fig:conicandline} we get that
    \[\mult_p^{\A^1}(\Gamma_1,\Gamma_2)+\mult_q^{\A^1}(\Gamma_1,\Gamma_2)=\langle a_{11}\rangle+\langle-a_{11}\rangle=h.\]
    In the right picture, there is only one intersection point $p$ and we computed in Example \ref{ex:enrichedintersectionmult} that $\mult_p^{\A^1}(\Gamma_1,\Gamma_2)=h$.
    So both sums of $\A^1$-intersection multiplicities agree and they also agree with $\frac{d_1\cdot d_2}{2}h=\frac{1\cdot 2}{2}h=h$ as they should by Corollary \ref{cor:tropicalenrichedbezout}.
\end{ex} 

Now Theorem \ref{thm:bezoutforcurvesoverpuiseux} follows directly from Theorem \ref{thm:enrichedintersectionmult} and Corollary \ref{cor:tropicalenrichedbezout}.

\begin{rmk}
    This also works in other dimensions, that is not only for curves but for hypersurfaces, and other Newton polytopes leading to a quadratically enriched Bernstein-Kushnirenko theorem as shown in \cite{JaramilloPuentesPauliBezout}.
\end{rmk}
\begin{rmk}
We can actually deduce McKean's Bézout's theorem (see Example \ref{ex:bezoutoveranyfield}) over any field $k$ of characteristic $0$ or bigger than the sum of degrees of the polynomials involved from Theorem \ref{thm:bezoutforcurvesoverpuiseux}. The reason for this is that the Poincaré Hopf Euler number (see Theorem \ref{thm:A1PHthm}) of the vector bundle $\glob(d_1)\oplus \glob(d_2)\rightarrow \P^2_k$ is sent to the Poincaré Hopf Euler number of $\glob(d_1)\oplus \glob(d_2)\rightarrow \P^2_{\Puiseux{k}}$ under the isomorphism $\GW(k)\xrightarrow{\cong} \GW(\Puiseux{k})$ (this isomorphism is induced by the inclusion $k\hookrightarrow \Puiseux{k}$ and inverse to the map in Exercise \ref{exercise:GWPuiseux}).
\end{rmk}

\subsection{Plane rational curve counts with point conditions}
Finally, we want to review a very powerful application of tropical geometry to enumerative geometry. %which allows us to solve a problem in enumerative geometry in $\GW(k)$ that, unlike Bézout's theorem, could yet not solved by other means.

A major breakthrough of the use of tropical geometry in enumerative geometry was Mikhalkin's tropical correspondence theorem \cite{Mikhalkin}.
Let's first describe the classical enumerative geometry problem.
This concerns counting rational plane degree $d$ curves through $n=3d-1$ points in general position over $\C$. Call this number $N_d$. 
\begin{ex}
    There is a unique degree $1$ curve, that is a line, going through two points. So $N_1=1$. 
    There is a unique conic, that is degree $2$ plane rational curve, through $5$ given point. So $N_2=1$.
    But after that $N_d$ grows fast. For example $N_3=12$, $N_4=620$, $N_5=87304$, $N_6=26312976$, $N_7=14616808192$.
\end{ex}

\begin{definition}
A plane tropical curve $\Gamma$ is \emph{nodal} if each of its vertices is $3$- or $4$-valent, and every $4$-valent vertex is dual to a parallelogram, meaning that locally the curve looks like a transverse crossing.
The genus of a plane tropical curve $\Gamma$ is given by
\[g(\Gamma)=b_1(\Gamma)-\#\{\text{$4$-valent vertices}\}\]
where $b_1(\Gamma)$ denotes the first Betti number of $\Gamma$.
    A plane tropical curve $\Gamma$ is \emph{rational} if its genus is $0$.
\end{definition}

Milhalkin assigns a multiplicity $\mult_\C(\Gamma)$ to a tropical curve $\Gamma$.
Mikhalkin's tropical correspondence theorem provides a way to find $N_d$ by translating the problem into a problem in tropical geometry with the following theorem. On the tropical side, the point configuration has to be in \emph{tropical general position} which is the tropical analogue of general position (for the precise definition see \cite{Mikhalkin}).
\begin{thm}[Mikhalkin]
\label{thm:mikhalkincomplexcorrespondence}
    $$N_d=N_d^{\trop}\coloneqq \sum_\Gamma\mult_\C(\Gamma)$$
    where the sum goes over all rational nodal tropical curves with Newton polygon $\Delta_d$ through a configuration of $3d-1$ points in tropical general position.
\end{thm}
\begin{rmk}
Mikhalkin's theorem is in fact stronger: it also applies to curves of positive genus. 
In particular, one can count plane genus~$g$ curves passing through $3d + g - 1$ points in general position in the plane, and Mikhalkin identifies this number with the weighted count of their tropical counterparts, using the same multiplicities as in Theorem~\ref{thm:mikhalkincomplexcorrespondence}. 
However, this extension does not hold over $\R$ or other non-algebraically closed fields, which is why we concentrate on rational curves in these lecture notes.

\end{rmk}

\begin{ex}
    There is a unique tropical line going through two given points in $\R^2$ if these points are in tropical general position (see Figure \ref{fig:tropicallinethrough2points}).
    \begin{figure}
    \begin{tabular}{ccc}
    \begin{tikzpicture}
    % Rays
    \draw[thick] (0,0) -- (-1.5,0);
    \draw[thick, red, fill=red] (-1,0) circle (2pt);
    \draw[thick] (0,0) -- (0,-1.5);
    \draw[thick, red, fill=red] (0,-1) circle (2pt);
    \draw[thick] (0,0) -- (1.5,1.5);
\end{tikzpicture}&
    \begin{tikzpicture}
    % Rays
    \draw[thick] (0,0) -- (-1.5,0);
    \draw[thick, red, fill=red] (-1,0) circle (2pt);
    \draw[thick] (0,0) -- (0,-1.5);
    \draw[thick] (0,0) -- (1.5,1.5);
    \draw[thick, red, fill=red] (1,1) circle (2pt);
\end{tikzpicture}&
    \begin{tikzpicture}
    % Rays
    \draw[thick] (0,0) -- (-1.5,0);
    \draw[thick] (0,0) -- (0,-1.5);
    \draw[thick, red, fill=red] (0,-1) circle (2pt);
    \draw[thick] (0,0) -- (1.5,1.5);
    \draw[thick, red, fill=red] (1,1) circle (2pt);
\end{tikzpicture}
    \end{tabular}
\caption{Tropical lines determined by two points in $\R^2$.\label{fig:tropicallinethrough2points}}
    \end{figure}
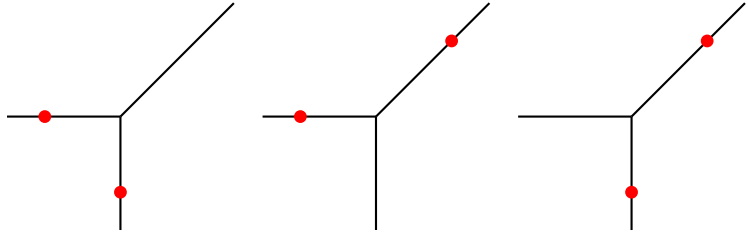
\end{ex}

\begin{rmk}
    If a point configuration of $3d-1$ points in $\R^2$ is ``tropically in general position", then all tropical rational degree $d$ curves through it are nodal.
\end{rmk}

Next consider the problem over $k=\R$. In this case, we are interested in the signed count of real rational degree $d$ plane curves through a configuration of $n_1$ real points and $n_2$ pairs of complex conjugate points with $n_1+2n_2=n=3d-1$.
Welschinger finds that one should count such curves with a sign depending on the type of real nodes of the curve, otherewise the count is not invariant meaning that it depends on the chosen point configuration. For this observe that a real curve can have two types of nodes:
\begin{enumerate}
        \item a \emph{split node} locally defined by $x^2-y^2=0$  \hspace{1cm}\scalebox{0.35}{\begin{tikzcd}
    \draw (-1,-1)--(1,1);
    \draw (1,-1)--(-1,1);
\end{tikzcd}}
        \item a \emph{solitary node} locally defined by $x^2+y^2=0$ \hspace{1cm}\scalebox{0.35}{\begin{tikzcd}
    \draw[dash pattern=on 1pt off 4pt] (-1,-1)--(1,1);
    \draw[dash pattern=on 1pt off 4pt] (1,-1)--(-1,1);
     \filldraw[black] (0,0) circle (2pt);
\end{tikzcd}}
\end{enumerate}
\begin{definition}
    Define the \emph{type} of a real node $z$ to be 
    \[\operatorname{type}(z)\coloneqq \begin{cases}
        +1& \text{if it is split}\\
        -1& \text{if it is solitary}
    \end{cases}\]
    Then the \emph{Welschinger sign} of a real cuve $C$ is given by
    \[\Wel_\R(C)\coloneqq \prod_{\text{real nodes $z$}}\type(z).\]
\end{definition}
Note that a real plane curve can also have complex nodes, but these do not contribute to $\Wel_{\R}(C)$.
Welschinger shows the following theorem \cite{Welschinger}.
\begin{thm}
Let $\mathcal{P}$ be a configuration of $n_1$ real points and $n_2$ pairs of complex conjugate points in general position with $n_1+2n_2=n=3d-1$. Then
    \[W_{d,n_2}\coloneqq\sum_{C} \Wel_\R(C)\]
    where the sum goes over all real rational degree $d$ curves through $\mathcal{P}$ is independent of the choice of point configuration.
\end{thm}
\begin{ex}
    For example $W_{1,n_2}=1$ and $W_{2,n_2}=1$. When $d=3$ we have $W_{3,n_2}=8-2n_2$ for $n_2=0,1,2,3,4$.
\end{ex}

Mikhalkin also provides a tropical correspondence theorem for the computation of $W_{d,0}$ (that is for point configurations consisting of only real points). Namely, he assigns a different multiplicity $\mult_\R(\Gamma)$ to a nodal tropical curve $\Gamma$ and shows the following theorem.
\begin{thm}[Mikhalkin \cite{Mikhalkin}]
\label{thm:Mikhalkinrealcorrespondence}
$$W_{d,0}=W_{d,0}^{\operatorname{trop}}\coloneqq\sum_\Gamma\mult_\R(\Gamma)$$
 where the sum goes over all rational nodal tropical curves with Newton polygon $\Delta_d$ through a configuration of $3d-1$ points in tropical general position.
\end{thm}
\begin{rmk}
\label{rmk:realandcomplexcorrespondence}
    The sums in Theorem \ref{thm:mikhalkincomplexcorrespondence} and Theorem \ref{thm:Mikhalkinrealcorrespondence} range over the same tropical curves, only the multiplicities differ.
\end{rmk}

For $W_{d,n_2}$ with $n_2\ge 0$ Shustin defines a more general multiplicity $\mult_\R(\Gamma)$ and provides a tropical correspondence theorem.
However, the tropical curves considered here are different.
\begin{thm}[Shustin \cite{Shustin}]
\label{thm:shustin}
    $$W_{d,n_2}=W_{d,n_2}^{\trop}=\sum\mult_\R(\Gamma)$$ where the sum goes over all rational tropical curves with Newton polygon $\Delta_d$ through a configuration of $n_1$ ``thin" (corresponding to real points) and $n_2$ ``fat" points (corresponding to complex points) in $\R^2$ in tropical general position. 
\end{thm}

Remark \ref{rmk:realandcomplexcorrespondence} already indicates that there should be a weighted count in $\GW(k)$ that specializes to both $N_d$ and $W_{d,0}$.
In fact, this works according to work of Levine \cite{LevineWelschinger} and recent work of Kass-Levine-Solomon-Wickelgren \cite{KassLevineSolomonWickelgren,KassLevineSolomonWickelgrenOrientation}.
We elaborate a little bit on the approach of Kass-Levine-Solomon-Wickelgren which expresses this weighted count of rational curves in $\GW(k)$ as the $\A^1$-degree of an evaluation map. 

Suppose $k$ is a perfect field of characteristic not equal to $2$ or $3$. 
    Let $\overline{\mathcal{M}}_{0,n}(\P^2_k,d)$ be the \emph{Kontsevich moduli space of $n$-marked genus $0$ stable maps to $\P^2$ of degree $d$}. That means the points correspond to $(f\colon C\rightarrow \P^2, x_1,\ldots,x_n)$ where $C$ is a nodal curve of arithmetic genus $0$, the $x_i\in C$ are smooth pairwise different points, the image of $C$ is a degree $d$ curve. Stable means that $(f\colon C\rightarrow \P^2, x_1,\ldots,x_n)$ has only finitely many automorphisms.
    Let $\sigma=(L_1,\ldots,L_r)$ be a list of finite field extensions $L_i/k$ such that $\sum_{i=1}^r[L_i:k]=n$. 
    There is a twisted evaluation map (see \cite[$\S5$]{KassLevineSolomonWickelgren})
    \[\ev_\sigma\colon (\overline{\mathcal{M}}_{0,n}(\P^2_k,d))_\sigma\rightarrow \prod_{i=1}^r\operatorname{Res}_{L_i/k}\P^2_k\]
    where $\operatorname{Res}_{L_i/k}\P^2_k$ is the restriction of scalars from $L_i$ to $k$.
    Now take a point configuration of $r$ points $p_1,\ldots,p_r$ in general position with residue field $\kappa(p_i)=L_i$ for $i=1,\ldots,r$. This defines a $k$-point in $\prod_{i=1}^r\operatorname{Res}_{L_i/k}\P^2_k$. Then in the preimage of this point configuration are exactly the stable maps with image a degree $d$ rational plane curve through this point configuration, which is exactly what we are after.
    
\begin{thm}[Kass-Levine-Solomon-Wickelgren \cite{KassLevineSolomonWickelgren}]
\label{thm:KLSW}
Suppose $k$ is a perfect field of characteristic not equal to $2$ or $3$ and $\sigma=(L_1,\ldots,L_r)$ be a list of finite field extensions $L_i/k$ such that $\sum_{i=1}^r[L_i:k]=n$. 
Then
    \[N^{\A^1}_{d,\sigma}\coloneqq \deg^{\A^1}\left(\ev_\sigma\colon \overline{\mathcal{M}}_{0,n}(\P^2_k,d)\rightarrow \prod_{i=1}^r\operatorname{Res}_{L_i/k}\P^2_k \right)\]
    is a well-defined element of $\GW(k)$. 
    Also,
    \[N^{\A^1}_{d,\sigma}=\deg^{\A^1}(\ev_\sigma)=\sum \deg^{\A^1}_u\ev_\sigma\]
    can be written as the sum of local $\A^1$-degrees over all stable maps that map to a curve passing through a given point configuration of $r$ points $p_1,\ldots,p_r$ in general position with residue field $\kappa(p_i)=L_i$ for $i=1,\ldots,r$.
\end{thm}
This is really non-trivial. To show that the $\A^1$-degree of this map is well-defined, one needs to show that it is \emph{relatively orientable} \cite{KassLevineSolomonWickelgrenOrientation}.

Recall that the Welschinger sign $\Wel_\R(C)$ depends on the type of real nodes. We want to generalize this.
\begin{definition}
    Let $C$ be a plane nodal curve over $k$ and let $z$ be a node of $C$. Then locally around $z$ the curve $C$ is given by an equation of the form $x^2-\alpha y^2=0$ when base changing to the residue field $\kappa(z)$ of $z$.
    The \emph{type} of $z$ is $\type(z)\coloneqq \alpha\in \nicefrac{\kappa(z)^\times}{(\kappa(z)^\times)^2}$.
    The \emph{quadratic weight} of $C$ is
    \[\Wel_k^{\A^1}(C)\coloneqq \prod_{\text{nodes $z$}}\langle N_{\kappa(z)/k} \type(z)\rangle \in \GW(k)\]
    where $N_{\kappa(z)/k}$ is the field norm.
\end{definition}
\begin{remark}
    For a real curve $C$ one has $\Wel^{\A^1}_\R(C)=\langle\Wel_\R(C)\rangle$ in $\GW(\R)$.
\end{remark}

The next theorem identifies the local $\A^1$-degrees in Theorem \ref{thm:KLSW} with the quadratic weight of the image of the stable map.
\begin{thm}[Kass-Levine-Solomon-Wickelgren \cite{KassLevineSolomonWickelgren}]
Suppose $k$ is a perfect field of characteristic not equal to $2$ or $3$ and $\sigma=(L_1,\ldots,L_r)$ be a list of finite field extensions $L_i/k$ such that $\sum_{i=1}^r[L_i:k]=n$. 
    \[\deg_u^{\A^1}\ev_\sigma=\Tr_{\kappa(u)/k}\left(\Wel_{\kappa(u)}^{\A^1}(u(C))\right)\]
    where $\kappa(u)$ is the residue field of the stable map $u\colon C\rightarrow \P^2$ in $\overline{\mathcal{M}}_{0,n}(\P^2,d)$. 
\end{thm}
    Consequently,
    \[N_{d,\sigma}^{\A^1}=\sum \Tr_{\kappa(u)/k}\left(\Wel_{\kappa(u)}^{\A^1}(u(C))\right)\]
    where the sum goes over all rational degree $d$ plane curves through a configuration of $r$ points in general position with residue field $L_1,\ldots,L_r$.
\begin{ex}
    For $k=\R$ and $\sigma=(\C,\ldots,\C,\R,\ldots,\R)$ consisting of $n_2$ times $\C$ and $n_1$ times $\R$ with $n_1+2n_2=n$ we have that
    \begin{align*}
    N_{\sigma,d}^{\A^1}&=\sum_{\text{real curves C}}\Wel_\R^{\A^1}(C)+\sum_{\text{complex curves $C$}}\Tr_{\C/\R}(\Wel_\C^{\A^1}(C))\\
    &=\sum_{\text{real curves C}}\Wel_\R^{\A^1}(C)+\sum_{\text{complex curves $C$}}\Tr_{\C/\R}(\langle1\rangle)\\
    &=\sum_{\text{real curves C}}\Wel_\R^{\A^1}(C)+\sum_{\text{complex curves $C$}}h
    \end{align*}
    which has signature
    \begin{align*}
        \sgn(N_{\sigma,d}^{\A^1})&=\sgn\left(\sum_{\text{real curves C}}(\Wel_\R^{\A^1}(C))+\sum_{\text{complex curves $C$}}h\right)\\
        &=\sgn\left(\sum_{\text{real curves C}}\Wel_\R(C)\right)=W_{d,n_2}
    \end{align*}
\end{ex}
Now let's return to the computation of $N_{d,\sigma}^{\A^1}$. Let's first consider $\sigma=(k,\ldots,k)$ with $k$ a perfect field of characteristic not equal to $2$ or $3$. Then one can define
\[\mult^{\A^1}_k(\Gamma)\coloneqq \begin{cases}
    \frac{\mult_{\C}(\Gamma)-1}{2}h+\langle \mult_\R(\Gamma)\rangle & \text{if $\mult_\R(\Gamma)\neq 0$}\\
    \frac{\mult_{\C}(\Gamma)}{2}h & \text{if $\mult_\R(\Gamma)=0$}
\end{cases}\]
which makes sense since $\mult_\R(\Gamma)\in \{-1,0,1\}$ and $\mult_\R(\Gamma)=0$ if and only if $\mult_\C(\Gamma)$ even.
There is a tropical correspondence theorem for $\sigma=(k,\ldots,k)$.
\begin{thm}[Jaramillo Puentes-Pauli \cite{JaramilloPuentesPauliCorrespondence}]
\label{thm:JPPcorrespondence}
Let $k$ be a perfect field of characteristic $0$ or greater than $d$.
For $\sigma=(k,\ldots,k)$ one has
\[N_{d,\sigma}^{\A^1}=N_{d,\sigma}^{\trop}\coloneqq \sum_\Gamma\mult_k^{\A^1}(\Gamma)\]
    where the sum goes over all rational nodal tropical curves with Newton polygon $\Delta_d$ through a configuration of $3d-1$ points in general position.
\end{thm}
A direct corollary is the following.
\begin{cor}
Let $k$ be a perfect field of characteristic $0$ or big.
For $\sigma=(k,\ldots,k)$ one has
    \[N_{d,\sigma}^{\A^1}=\frac{N_d-W_{d,0}}{2}h+W_{d,0}\langle1\rangle\]
in $\GW(k)$.
\end{cor}
\begin{rmk}
    The sum in Theorem \ref{thm:JPPcorrespondence} is over the same tropical curves as in Theorem \ref{thm:mikhalkincomplexcorrespondence} and Theorem \ref{thm:Mikhalkinrealcorrespondence}. 
\end{rmk}
What about other $\sigma$? Are there also tropical correspondence theorems similar to Shustin's theorem \ref{thm:shustin}?
Yes, due to recent work of Jaramillo Puentes-Markwig-Pauli-Röhrle \cite{JPMPRnew} for when $\sigma=(k(\sqrt{d_1}),\ldots,k(\sqrt{d_{n_2}}),k,\ldots,k)$ consists of quadratic and trivial field extensions. 
\begin{thm}[Jaramillo Puentes-Markwig-Pauli-Röhrle \cite{JPMPRnew}]
Let $k$ be a perfect field of characteristic $0$ or greater than $d$ involved 
    and let $\sigma=(k(\sqrt{d_1}),\ldots,k(\sqrt{d_{n_2}}),k,\ldots,k)$. Then
    \[N_{\sigma,d}^{\A^1}=N_{\sigma,d}^{\A^1,\trop}\coloneqq\sum_\Gamma\mult_k^{\A^1}(\Gamma)\]
    where the sum goes over all rational tropical curves through a configuration of $n_1$ ``thin" (corresponding to $k$-points) and $n_2$ ``fat" points (corresponding to $k(\sqrt{d_i})$-points) in $\R^2$ in tropical general position. 
\end{thm}
\begin{rmk}
    For $k=\R$ and $\sigma=(\C,\ldots,\C,\R,\ldots,\R)$, taking signatures in $N_{d,\sigma}^{\A^1,\trop}=\sum_\Gamma \mult_k^{\A^1}(\Gamma)$ recovers Shustin's tropical correspondence theorem \ref{thm:shustin}.
\end{rmk}

The techniques in \cite{JPMPRnew} in principle also work for general $\sigma$, the hard task is to work out the multiplicities for this general setting.
\medskip
Why is this useful?
Mikhalkin showed that counting plane tropical curves can be reduced to counting \emph{lattice paths} which can all be found by purely using combinatorics \cite{Mikhalkin}. Gathmann--Markwig gave a tropical proof of the Caporaso--Harris recursion, a relation computing all \(N_d\) from \(N_1\)~\cite{GathmannMarkwigCH}. This was also generalized to the $\GW(k)$ setting \cite{JPMPRfirst}. The problem of counting tropical curves can also be translated into counting \emph{floor diagrams}, a problem that admits explicit algorithms. This method extends to a generalization with multiplicities valued in \(\GW(k)\); see~\cite{JPMPRnew}. In short, tropical geometry yields direct, implementable algorithms for computing curve-counting invariants. There is a recent survey on how to use combinatorial tools like lattice paths for these tropical curve counting problems \cite{JPMRPsurvey}.

\bibliographystyle{alpha}
\bibliography{literature}

\end{document}